\documentclass[reqno]{amsart}
\usepackage[utf8]{inputenc}
\usepackage{amssymb, amsmath, amsthm,bbm, color, enumerate, tikz, tikz-3dplot}

\usepackage{hyperref}
 \hypersetup{ hidelinks = true } 

\usepackage{makecell}
\setcellgapes{2pt}

\tdplotsetmaincoords{75}{15}

\usepackage[thinlines]{easytable}
\usetikzlibrary{patterns}

\usepackage{comment}

\numberwithin{equation}{section}

\newcommand{\R}{\mathbb{R}}
\newcommand{\Z}{\mathbb{Z}}
\newcommand{\N}{\mathbb{N}}

\newcommand{\ud}{\, \mathrm{d}}
\newcommand{\BL}{\mathrm{BL}}

\DeclareMathOperator{\I}{I}
\DeclareMathOperator{\II}{II}

\newcommand{\supp}{\mathrm{supp}\,}

\def\bbone{{\mathbbm 1}}

\theoremstyle{plain}
\newtheorem{theorem}{Theorem}[section]
\newtheorem{lemma}[theorem]{Lemma}
\newtheorem{proposition}[theorem]{Proposition}
\newtheorem{conjecture}[theorem]{Conjecture}

\newtheorem*{key estimate}{Key estimate}

\newtheorem*{hypothesis}{Induction Hypothesis}

\theoremstyle{remark}
\newtheorem*{remark}{Remark}
\newtheorem{remarkn}[theorem]{Remark}

\theoremstyle{definition}
\newtheorem{definition}[theorem]{Definition}

\newtheorem*{key example}{Key example}

\title[$L^p-L^q$ local smoothing via $k$-broad Fourier restriction]{$L^p-L^q$ local smoothing estimates for the wave equation via $k$-broad Fourier restriction}
\author{David Beltran}
\author{Olli Saari}
\date{\today}


\address{David Beltran: Department of Mathematics, University of Wisconsin, 480 Lincoln Drive, Madison, WI, 53706, USA. \url{https://orcid.org/0000-0001-6711-9741}}
\email{dbeltran@math.wisc.edu} 
	
\address{Olli Saari: Mathematical Institute, 
	University of Bonn,
	Endenicher Allee 60, 53115, Bonn,
	Germany. \url{https://orcid.org/0000-0003-1212-8100}}
	\email{saari@math.uni-bonn.de}

\begin{document}

\begin{abstract}
We explore the connection between $k$-broad Fourier restriction estimates and sharp regularity $L^p-L^q$ local smoothing estimates for the solutions of the wave equation in $\R^{n}\times \R$ for all $n \geq 3$ via a Bourgain--Guth broad-narrow analysis. An interesting feature is that local smoothing estimates for $e^{i t \sqrt{-\Delta}}$ are not invariant under Lorentz rescaling.
\end{abstract}

\maketitle

\section{Introduction}

Let $u$ denote the solution of the Cauchy problem for the wave equation in $\R^{n} \times \R$
\begin{equation*}
    \left\{ \begin{array}{l}
        (\partial_t^2 - \Delta)u(x,t) = 0  \\[3pt]
        u(x,0)  := f(x), \quad \partial_t u(x,0)  := 0.
    \end{array}
    \right.
\end{equation*}
It is well known that $u$ can be written in terms of the half-wave propagator
$$
e^{i t \sqrt{-\Delta}} f(x) := \frac{1}{(2\pi)^n} \int_{\R^n} e^{i (x \cdot \xi +  t |\xi|)} \widehat{f} (\xi)  \ud \xi,
$$
which satisfies the \textit{fixed-time} bounds  \cite{Peral1980, MiyachiWave}
\begin{equation}\label{wave propagator}
\| e^{i t \sqrt{-\Delta}} f \|_{L^p_{-\bar s_p}(\R^n)} \, \lesssim_t \, \| f \|_{L^p(\R^n)}\, ,  \qquad  \bar{s}_p:=  (n-1) \Big|\frac{1}{2}- \frac{1}{p}\Big| \,,
\end{equation}
for any $1 < p < \infty$ and any $t>0$, where the implicit constant is locally bounded in $t$. 
Here $L_s^{p}$ denotes the Bessel potential space,
and we refer to the end of the introduction for the rest of the notation. 
Whilst these bounds are sharp for each fixed $t$, Sogge \cite{Sogge91} observed that there exists some $\sigma>0$ such that
\begin{equation}\label{local smoothing conj}
\Big( \int_1^2 \| e^{i t \sqrt{-\Delta}} f \|_{L^p_{-\bar s_p + \sigma}(\R^n)}^p  \ud t \Big)^{1/p} \lesssim \| f \|_{L^p(\R^n)}
\end{equation}
holds for all $2 < p < \infty$. This regularity gain in $L^p$ satisfied by $e^{it \sqrt{-\Delta}}$ after a local integration in time is commonly referred to as the \textit{local smoothing phenomenon} of the wave equation. It is conjectured \cite{Sogge91} that \eqref{local smoothing conj} holds for all $\sigma < \sigma_p$ where
\begin{equation*}
    \sigma_p:= \begin{cases}
    \begin{array}{ll}
        1/p & \text{ if } \quad  \frac{2n}{n-1} \leq  p < \infty,  \\
        \bar s_p & \text{ if } \quad  \quad \,\, 2 < p \leq \frac{2n}{n-1}.
    \end{array}
    \end{cases}
\end{equation*}
The \textit{local smoothing conjecture}\footnote{It is also expected that endpoint regularity results with $\sigma=1/p$ should hold if $p>2n/(n-1)$: see \cite{HNS2011} for results in this direction if $n \geq 4$. Similarly, the forthcoming Conjecture \ref{conjecture:LpLq} could hold for endpoint regularity cases $\sigma=\sigma_{p,q}$. Such endpoint cases will not be considered in this paper.} %: see \cite{Lee2003} for results if $n=2$.}  
is at its strongest when $p=\frac{2n}{n-1}$. 
The remaining cases follow by interpolation against the fixed-time estimates \eqref{wave propagator}. 
More precisely, one interpolates \eqref{local smoothing conj} with the energy conservation identity
\begin{equation}\label{eq:energy}
\| e^{it\sqrt{-\Delta}} f \|_{L^2(\R^{n} \times [1,2])} = \| f \|_{L^2(\R^n)} 
\end{equation}
and the $L^\infty$ estimate (see for instance \cite[Chapter IX, \S4]{bigStein})
\begin{equation}\label{infty endpoint}
    \| e^{it\sqrt{-\Delta}} f \|_{L^\infty_{- (n-1)/2 - \varepsilon}(\R^{n} \times [1,2])} \lesssim_t \| f \|_{L^\infty(\R^n)},
\end{equation}
which holds for all $\varepsilon>0$.

The local smoothing conjecture has been studied 
in numerous papers ever since it was first posed in \cite{Sogge91}, see for instance \cite{Mockenhaupt1992,Wolff2000,Laba2002,Garrigos2009,Garrigos2010,HNS2011,Lee2013,BD2015,LeeLS}. When $n=2$, sharp results follow by the work of Guth, Wang and Zhang \cite{GWZ}.
They prove a sharp version of a reverse square function estimate considered by Mockenhaupt \cite{Mock93},
which then implies the conjecture by a slight modification of the methods of \cite{Mockenhaupt1992} and an application of Córdoba's sectorial square function \cite{Cordoba82}: see \cite[\S 6]{TV2}. 
When $n \geq 3$, the  conjecture holds for all $p \geq \frac{2(n+1)}{n-1}$ by the Bourgain--Demeter decoupling theorem \cite{BD2015} and the method of Wolff \cite{Wolff2000}.
See also \cite{GLMX} for partial results in the range $2 \leq p \leq \frac{2(n+1)}{n-1}$. 
Verification of the full local smoothing conjecture would imply affirmative answers to a number of other important open problems such as the Bochner--Riesz conjecture, the Fourier restriction conjecture and the Kakeya conjecture; see \cite{TaoBR} for further background.

This note focuses on an $L^p-L^q$ variant of the local smoothing conjecture. 
The fixed-time estimate \cite{Littman, Strichartz1970, Brenner}
\begin{equation}\label{eq:L1Linfty}
\| e^{i t \sqrt{-\Delta}} f \|_{L^\infty_{-(n+1)/2 - \varepsilon} (\R^n)} \lesssim_t \|f \|_{L^1(\R^n)}, \quad \varepsilon > 0,
\end{equation}
along with the complex interpolation method can be used to upgrade  \eqref{wave propagator} to the $L^p$-improving inequality
\begin{equation}\label{eq:LpLq fixed time}
\| e^{i t \sqrt{-\Delta}} f \|_{L^q_{-\bar s_{p,q}}(\R^n)} \, \lesssim_t  \, \| f \|_{L^p(\R^n)}\, ,  \quad  \bar{s}_{p,q}:=\left
\lbrace \begin{array}{ll}
        \bar s_q + \frac{1}{p} - \frac{1}{q}  & \text{ if } \quad  q \geq p' \\
        \bar s_p + \frac{1}{p} - \frac{1}{q}  & \text{ if } \quad  q \leq p'
    \end{array} \right . \,,
\end{equation}
valid for any $1 < p \leq q < \infty$ and any $t >0$, and where the implicit constant is locally bounded in $t$. Here $p'=p/(p-1)$. Similarly, any local smoothing estimate \eqref{local smoothing conj} can be interpolated with \eqref{eq:L1Linfty} to obtain $L^p(\R^n)- L^q(\R^n \times [1,2])$  estimates for $q\geq p$. This motivates the following conjecture \cite{SS1997,TV2}.

\begin{conjecture}[$L^p-L^q$ local smoothing conjecture]\label{conjecture:LpLq} For $n \geq 2$, the inequality
\begin{equation}\label{LpLq conjecture}
\Big(\int_1^2 \| e^{i t \sqrt{-\Delta}} f \|_{L^q_{-\bar{s}_{p,q} + \sigma}(\R^n)}^q  \ud t \Big)^{1/q} \lesssim \| f \|_{L^p(\R^n)}
\end{equation}
holds for all $\sigma < \sigma_{p,q}$ if $1< p \leq q < \infty$ and $p' < q$, where
\begin{equation*}
    \sigma_{p,q}:= \left \lbrace \begin{array}{ll}
        \frac{1}{q} & \text{ if } \quad  \frac{1}{q} \leq \frac{n-1}{n+1} \frac{1}{p'} %, \quad \frac{2n}{n-1} \leq q<\infty \\
        \\
        \frac{(n-1)}{2}\left(\frac{1}{p'} - \frac{1}{q} \right) & \text{ if } \quad  \frac{1}{q} \geq \frac{n-1}{n+1} \frac{1}{p'}%, \quad  2 \leq p \leq \frac{2n}{n-1}
    \end{array} \right . .
\end{equation*}
\end{conjecture}

By the preceding discussion, 
validity of the conjecture for $q=p$ implies the cases $q > \max\{p,p'\}$ by interpolation with \eqref{eq:L1Linfty}; thus Conjecture \ref{LpLq conjecture} is a consequence of the $L^p-L^p$ result in \cite{GWZ} when $n=2$. When $q > p$, Conjecture \ref{conjecture:LpLq} is at its strongest on the \textit{critical line} 
\[\frac{1}{q}=\frac{n-1}{n+1}\frac{1}{p'};\] 
validity of the sharp regularity estimates \eqref{LpLq conjecture} for a pair $(p^*,q^*)$ there immediately implies, by interpolation with \eqref{eq:energy}, \eqref{infty endpoint} and \eqref{eq:L1Linfty}, sharp regularity $L^p-L^q$ local smoothing estimates for $(1/p,1/q) \in \mathfrak{Q}_{p^*,q^*} \backslash (\overline{P_0P_1}\cup\overline{P_1P_2})$, where $\mathfrak{Q}_{p^*,q^*}$ is the closed quadrangle with vertices
\begin{equation*}
P_0=(0,0), \qquad  P_1=(1,0), \qquad P_2=(1/2,1/2), \qquad P_*=(1/p^*,1/q^*).
\end{equation*}
This region of validity can further be extended to a hexagon with additional vertices at $(1/p_1,1/p_1)$ and $(1/p_2,1/p_2)$ if the conjecture is known to hold on the line $p=q$ for some $\frac{2n}{n-1} < p_1 < \infty$, $2 < p_2 < \frac{2n}{n-1}$.

It is of fundamental importance
that sharp regularity estimates on the critical line
can be obtained despite the full conjecture being open.
As a prime example we mention the Strichartz estimate \cite{Strichartz77}  
\begin{equation}\label{Strichartz cone}
\| e^{i t \sqrt{-\Delta}} f \|_{L^{\frac{2(n+1)}{n-1}} (\R^{n+1})} \lesssim \| f \|_{\dot{L}_{\frac{1}{2}}^{2}(\R^n)},
\end{equation}
which corresponds to the endpoint case $\sigma=\sigma_{p,q}$ in \eqref{LpLq conjecture} on the critical line for $q = \frac{2(n+1)}{n-1}$. Sharp $L^p-L^q$ local smoothing estimates  beyond \eqref{Strichartz cone} were first studied by Schlag and Sogge \cite{SS1997} when $n=2$.
Further improvements beyond $q=\frac{2(n+1)}{n-1}$ and in any dimension $n \geq 2$ were obtained in \cite{TV2, Lee2003, Lee2006} using the Wolff--Tao bilinear Fourier restriction estimates \cite{WolffCone, TaoCone} for the cone, which can be interpreted as local smoothing estimates via Plancherel's theorem. These bilinear estimates and their conjectured $k$-linear counterparts are of the form
\begin{equation}\label{eq:multilinear}
    \| \prod_{j=1}^k |e^{i t \sqrt{-\Delta}} f_j|^{1/k}\|_{L^{p}(B_R)} \lesssim_\varepsilon R^\varepsilon \prod_{j=1}^k \| f_j \|_{L^2(\R^n)}^{1/k}, \quad \textrm{$p\geq \bar p_{n,k} := \frac{2(n+k+1)}{n+k-1}$}
\end{equation}
where $2 \leq k \leq n+1$; $\supp \widehat{f}_j \subseteq \{\xi \in \R^n : 1 \leq |\xi| \leq 2 \}$ for all $1 \leq j \leq k$; the sets $\{\frac{\xi}{|\xi|} : \xi \in \supp \widehat{f}_j \}$ are separated; $B_R \subseteq \R^{n+1}$ denotes a ball of radius $R$
and the estimates are understood to hold for all $\varepsilon>0$ and all $R\ge 1$.
The only known cases are $k=2$ \cite{WolffCone, TaoCone}, 
$k=n$ \cite{Bejenaru2020} 
and $k=n+1$ \cite{BCT}. 
The remaining cases $3 \leq k < n$ are open
up to some partial positive results for $p \geq \frac{2k}{k-1}$ \cite{BCT}.

As the exponents $\bar p_{n,k}$ decrease with $k$, it is natural to explore if higher orders of multilinearity imply further progress on Conjecture \ref{conjecture:LpLq}. 
This line of investigation was considered by Lee \cite{LeeLS} for $n=2$ using 
the trilinear reduction of Lee and Vargas \cite{Lee2012}; see also the recent work \cite{HKL2021}. In this note, we further extend the multilinear approach to any dimension and any level of linearity in the case $q > p$. 
We remark that whereas partial results for $q=p$ using this method
were discussed in \cite{GLMX},
our focus is on sharp results with $q>p$.
Rather than working with $k$-linear estimates, we will work with their $k$-\textit{broad} variants (see \S\ref{sec:kbroad}), which hold in the full range $p \geq \bar p_{n,k}$. The idea of substituting the missing $k$-linear estimates by $k$-broad estimates goes back to the work of Guth \cite{Guth14} on Fourier restriction estimates for the paraboloid. Analogous results for conic surfaces have recently been obtained in \cite{OW,GLMX,Schippa}.  

We use a by now standard broad-narrow analysis from \cite{Guth16,BG}.
To do so,
we cannot restrict attention to the half-wave propagator $e^{i t\sqrt{-\Delta}}$ but we are forced to consider a larger class of operators that remains closed under Lorentz rescaling (see \S \ref{sec:Lorentz}).
Note that, unlike Fourier restriction estimates for the cone, local smoothing estimates for $e^{i t\sqrt{-\Delta}}$ are not invariant under Lorentz rescaling as they are not invariant under rotations in $\R^n \times \R$. As a consequence, one needs to use the $k$-broad estimates for perturbations of the light cone from \cite{GLMX,Schippa} instead of only those for the light cone from \cite{OW}.

\begin{figure}\label{figure:LpLq}
\includegraphics[scale=0.75]{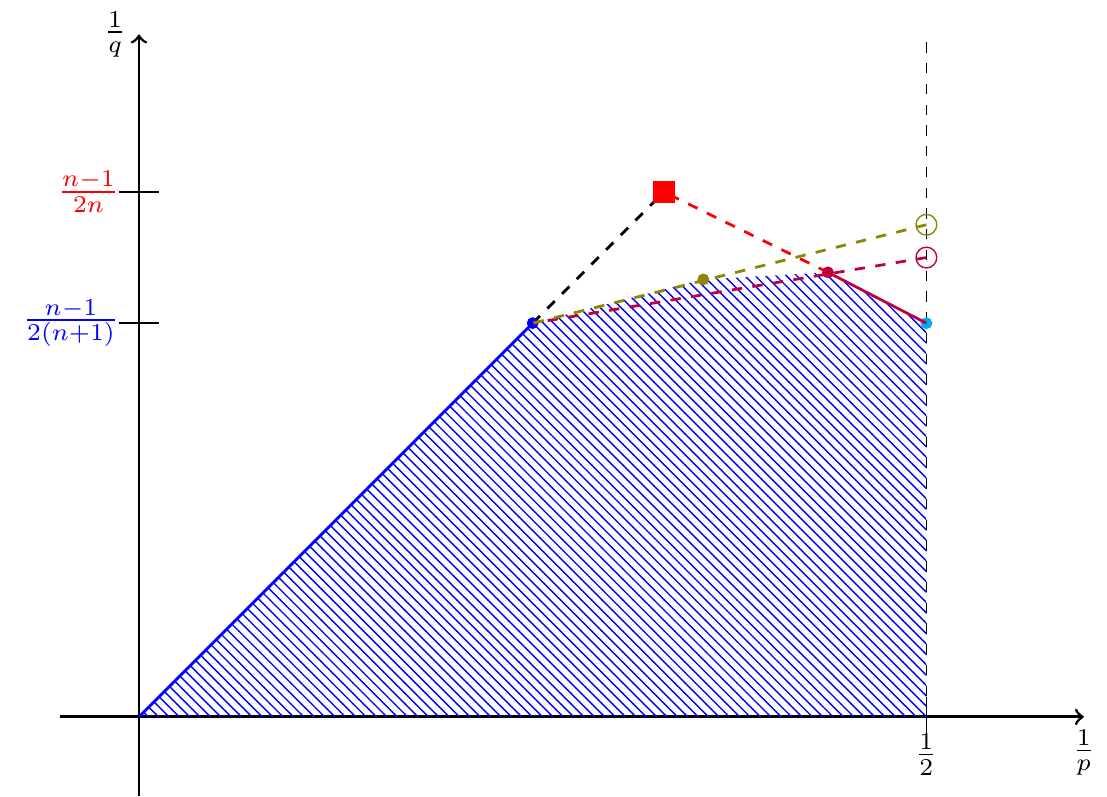}
\caption{$L^p-L^q$ local smoothing estimates for all $\sigma< \sigma_{p,q}$ for $\frac{1}{q} \leq \frac{n-1}{n+1}\frac{1}{p'}$, $p \geq 2$ hold in the shaded region $\mathfrak{P}_n$. The critical point of the local smoothing conjecture is depicted as a red square, and the descending red dashed line is the critical line of the $L^p-L^q$ conjecture. The dark blue point follows from the decoupling theorem \cite{BD2015}. The hollow circles denote $k$-broad restriction estimates in \cite{OW, GLMX} $k = 3,4$. The estimates at the purple and olive points are the content of our Theorem \ref{thm:smoothing all k}. Higher degrees of multilinearity imply further points, but they have been left out from the picture for clarity.}
\end{figure}

Our first result is a sharp $L^p-L^q$ local smoothing estimate on the critical line $\frac{1}{q}=\frac{n-1}{n+1} \frac{1}{p'}$. For future applicability, we state the next theorem in terms of the best exponent for which there is sharp regularity results in Conjecture \ref{conjecture:LpLq} when $q=p$. Such a statement requires the aforementioned larger class of operators, which we introduce in what follows. Let $\Phi_{\mathrm{conic}}^+$ denote the class of functions $\varphi:\R^n \to \R$ smooth away from $0$, homogeneous of degree $1$ and satisfying that $\partial_{\xi \xi}^2 \varphi (\xi)$ has $n-1$ positive eigenvalues on $\R^n \backslash \{0\}$. Given $\varphi \in \Phi_{\mathrm{conic}}^+$, define the wave-like propagator
$$
U_{\varphi} f (x,t):= \frac{1}{(2\pi)^n} \int_{\R^n} e^{i ( x \cdot \xi + t\varphi (\xi))} \widehat{f}(\xi) \ud \xi.
$$
Note that a standard computation reveals that $\varphi(\xi)=|\xi| \in \Phi^+_{\mathrm{conic}}$ and thus $e^{i t \sqrt{-\Delta}}$ is of the form $U_\varphi$. It is well-known \cite{MiyachiWave, Brenner} that if $\varphi \in \Phi^+_{\mathrm{conic}}$, $U_{\varphi}$ continues to satisfy \eqref{wave propagator}, \eqref{eq:L1Linfty} and \eqref{eq:LpLq fixed time}, that is,
\begin{equation}\label{eq:LpLq fixed time Us}
\| U_\varphi f \|_{L^q_{-\bar s_{p,q}}(\R^n)} \, \lesssim_t  \, \| f \|_{L^p(\R^n)}
\end{equation}
for $1 < p \leq q < \infty$. One can formulate Conjecture \ref{conjecture:LpLq} for $U_\varphi$, which is also known to hold for $q= p \geq \frac{2(n+1)}{n-1}$ by \cite{BD2015}.

\begin{theorem}\label{thm:LpLq}
Let $n \geq 2$ and let $\bar{p}_n \geq 2n/(n-1)$ be the smallest number $p$ for which
\begin{equation*}%\label{eq:general phase LS pp}
    \| U_{\varphi} f \|_{L^p_{-\bar s_p + \sigma }(\R^n \times [1,2])} \lesssim  \| f \|_{L^p(\R^n)}
\end{equation*}
holds for all $\sigma < 1/ p$ and all $\varphi \in \Phi_{\mathrm{conic}}^+$. Then Conjecture \ref{conjecture:LpLq} holds for all $1 < p \leq q < \infty$ satisfying
\[q \geq  \frac{2  \bar p_n \left( n^2 + 3n-1\right)-4 n (n+4)}{(n-1)   ( (n+2) \bar p_n - 2(n + 3))  }, \qquad \frac{1}{q} = \frac{n-1}{n+1} \frac{1}{p'}, \]
and, more generally, 
\begin{equation*}
    \| U_{\varphi} f \|_{L^q_{-\bar s_{p,q} + \sigma }(\R^n \times [1,2])} \lesssim  \| f \|_{L^p(\R^n)}
\end{equation*}
holds for all $\sigma < 1/q$, $p$ and $q$ as above and all $\varphi \in \Phi_{\mathrm{conic}}^+$.

In particular, as $\bar{p}_n \leq 2(n+1)/(n-1)$ by \cite{BD2015}, then Conjecture \ref{conjecture:LpLq} holds for 
\[q \geq \frac{2(n^2+6n-1)}{(n-1)(n+5)}, \quad \frac{1}{q}=\frac{n-1}{n+1} \frac{1}{p'}.\]
\end{theorem}

For $n \geq 3$ our results are an improvement over the estimates obtained by bilinear methods in \cite{Lee2006}, which implied Conjecture \ref{conjecture:LpLq} for $q \geq \frac{2(n^2+2n-1)}{(n-1)(n+1)}$, $\frac{1}{q}=\frac{n-1}{n+1} \frac{1}{p'}$.

Theorem \ref{thm:LpLq} is proved using $3$-broad estimates only. 
The use of higher degrees of multilinearity in the proof 
would cause the method to become increasingly inefficient on the critical line. 
However, higher orders of multilinearity can be used away from the critical line. 
This is the content of our next theorem, from which Theorem \ref{thm:LpLq} follows by setting $k=3$.

\begin{theorem}\label{thm:smoothing all k}
Let $n \geq 2$ and let $\bar{p}_n \geq 2n/(n-1)$ be the smallest number $p$ for which
\begin{equation*}
    \| U_{\varphi} f \|_{L^p_{-\bar s_p + \sigma }(\R^n \times [1,2])} \lesssim  \| f \|_{L^p(\R^n)}
\end{equation*}
holds for all $\sigma < 1/ p$ and all $\varphi \in \Phi_{\mathrm{conic}}^+$.
Then Conjecture \ref{conjecture:LpLq} holds for all pairs $(\bar p(k), \bar q(k))$,
\begin{align*}
    \bar p(k) &= \frac{ 2 \bar p_n \left(2n^2 +k(n+4) -k^2+3 n-5\right)-4 \left(n+k+1\right)\left(2n-k+3\right)}{  \bar p_n \left(n+k+1\right)\left(2n-k+1\right) -2 \left(2n^2 +k(n-2) -k^2 +5 n+9\right)}
, \nonumber \\
\bar q(k) &= \frac{ 2 \bar p_n \left(2n^2 +k(n+4) -k^2+3 n-5\right)-4 \left(n+k+1\right)\left(2n-k+3\right)}{ \bar p_n \left(n+k-1\right)\left(2n-k+1\right) -2 \left( 2n^2 +kn -k^2+n+3\right)}.
\end{align*}
with $k \in \{2, \ldots, n+1\}$. 

Furthermore, Conjecture \ref{conjecture:LpLq} holds for all
$(1/p,1/q) \in (\mathfrak{P}_n \cup \mathfrak{T}_n) \backslash (\overline{P_0P_1}\cup\overline{P_1P_2})$
where $\mathfrak{P}_n$ is the convex hull of
\[ (1/\bar p(k), 1/\bar q(k)), \quad (1/\bar p_n, 1/\bar p_n), \quad P_0=(0,0), \quad P_1=(1,0),
\]
see Figure \ref{figure:LpLq}.
The set $\mathfrak{T}_n$ is the triangle formed by 
\[(1/\bar p(3), 1/\bar q(3)), \quad P_1 = (1,0), \quad P_2=(1/2,1/2).\]
More generally,
\begin{equation*}
    \| U_{\varphi} f \|_{L^q_{-\bar s_{p,q} + \sigma }(\R^n \times [1,2])} \lesssim  \| f \|_{L^p(\R^n)}
\end{equation*}
holds for all $\sigma < 1/q$, $p,q$ as above and all $\varphi \in \Phi_{\mathrm{conic}}^+$.

In particular, as $\bar{p}_n \leq 2(n+1)/(n-1)$ by \cite{BD2015}, we have the values
$$
    \bar p(k)=\frac{2(n^2+2nk-k^2+3k-1)}{n^2+2nk-k^2-k+5} \quad \text{and} \quad \bar q(k)=\frac{2(n^2+2nk-k^2+3k-1)}{n^2+2nk-k^2+k-2n+1}.
$$
\end{theorem}

We remark that the sharp regularity estimates from Theorem \ref{thm:smoothing all k} can be interpolated against any current non-sharp regularity $L^p-L^p$ local smoothing estimates (such as those in \cite{GLMX}) to obtain partial results in the exterior of $\mathfrak{T}_n$.

We finish the introduction with a contextual remark. 
One of our original motivations was to investigate how close the state of art in $L^p-L^q$ local smoothing estimates for $e^{i t \sqrt{-\Delta}}$ is from solving a problem that was left open in our earlier joint work with Ramos \cite{BRS2018}: Let $n = 4$ and let $\sigma$ be the normalised surface measure of the unit sphere in $\R^{n}$. Does there exist $p \in (1,\infty)$ and $\alpha \in [1,n-1)$ such that 
\[f \mapsto \sup_{t > 0} |t^{\alpha} \sigma_t * f| \]
maps $L^{p}$ to a first order Sobolev space? The question has a positive answer if sharp $L^{p}- L^{q}$ local smoothing estimates hold for $q \geq 3 - 1/6 - \epsilon $. Using the best known estimates in Theorem \ref{thm:LpLq}, we only get sharp local smoothing for $q \geq 3- 1/9$ and hence miss the threshold by $1/18$.

\subsubsection*{Structure of the paper} We begin by making some standard reductions in \S \ref{sec:initial}, which reduce Theorem \ref{thm:smoothing all k} to the upcoming Theorem \ref{prop:main prop LpLq}, which is a local estimate for functions with compact Fourier support that is well separated from the origin. In \S \ref{sec:Lorentz} we address the Lorentz rescaling, which is a main ingredient in the proof of Theorem \ref{thm:smoothing all k} and the reason for introducing the class of phase functions $\Phi_{\mathrm{conic}}^+$. The concept of a $k$-broad norm is recalled in \S \ref{sec:kbroad} and in \S \ref{subsec:narrow decoupling} we present a narrow decoupling for the operators $U_\varphi$. The proof of Theorem \ref{thm:smoothing all k} is presented in \S \ref{sec:proof}.

\subsubsection*{Notation}
Given $R \geq 1$, $B_R^n$ denotes a ball of radius $R$ in $\R^n$ and $B_R$ denotes a ball of radius $R$ in $\R^{n} \times \R$. Given a measurable set $A \subset \R^{n+1}$, $A^c$ denotes its complementary set. The notation $A \lesssim B$ is used if $A \leq C B$ for some constant $C>0$. If the constant $C$ depends on a certain list of relevant parameters $L$, we use the notations $C=C_L$ and $\lesssim_L$. 
The case of dimension and the integrability parameters $p$ and $q$ may also be suppressed from the notation.
The relations $A \gtrsim_L B$ and $A \sim_L B$ are defined similarly. 

For a Schwartz function $f$,
we define the Fourier transform with the normalisation 
\[ \widehat{f}(\xi) = \int_{\R^n} e^{-i x \cdot \xi} f(x) \ud x . \]
As we are most of the time concerned with a priori estimates with Schwartz data,
we omit the standing assumption of a function $f$ being Schwartz in the statements of theorems and lemmata.
We only mention the assumed regularity if it is other than Schwartz.

A weight function adapted to a ball $B_R \subset \R^{n+1}$ of radius $R$ and centre $c$ is defined as
\[ w_{B_R}^N (z) = \Big( 1 + \Big(\frac{|z-c|}{R} \Big)^2 \Big)^{-N/2}, \]
where $N$ is a large dimensional constant. 

Given $1 < p \leq q < \infty$, $p'\leq q$, $0 < \bar \sigma \leq \sigma_{p,q}$ and $\varphi \in \Phi_{\mathrm{conic}}^+$, we say that there is $(p,q, \bar \sigma)$ \textit{local smoothing} for $U_\varphi$ or that a $(p,q,\bar \sigma)$ \textit{local smoothing estimate} for $U_\varphi$ holds if \begin{equation*}
    \| U_{\varphi} f \|_{L^q_{-\bar s_{p,q} + \sigma }(\R^n \times [1,2])} \lesssim  \| f \|_{L^p(\R^n)}
\end{equation*} holds for all $\sigma < \bar \sigma$. If $\bar \sigma=\sigma_{p,q}$, we say that there is \textit{sharp regularity $(p,q)$ local smoothing} for $U_\varphi$.

\subsubsection*{Acknowledgements} 
We would like to thank Chuanwei Gao, Bochen Liu, Changxing Miao, Jiqiang Zheng and Yakun Xi for communications regarding their related works  \cite{GMZ, GLMX} concerning the case $q=p$ and the current manuscript. D.B. was supported by the NSF grant DMS-1954479. O.S. was supported by the Deutsche Forschungsgemeinschaft (DFG, German Research Foundation) under Germany's Excellence Strategy – EXC-2047/1 – 390685813 as well as the SFB 1060. %This project was finalised during D.B.'s participation in the research program ``Interactions between Geometric measure theory, Singular integrals, and PDEs", hosted by the Hausdorff Institute of Mathematics.

\section{Initial reductions}\label{sec:initial}
Before proceeding with the proof of Theorem \ref{thm:smoothing all k} we perform some standard reductions which are useful for showing that there is $(p,q, \bar \sigma)$ local smoothing for $U_\varphi$.

\subsection{Dyadic decomposition}\label{subsec:dyadic}
Given $\varphi \in \Phi_{\mathrm{conic}}^+$, the first step is to break up the operator $U_\varphi$ into pieces which are Fourier supported on dyadic annuli. Let  $\zeta \in C_c^\infty(\R)$  with $\supp \zeta  \subseteq [1/2, 2]$ be such that $\sum_{k \in \Z} \zeta(2^{-k} r) = 1$ for all $r>0$. Define $\eta(\xi) = \zeta (|\xi|)$ for $\xi \in \R^n$. Thus,
$$
U_\varphi f (x,t) = U_\varphi (\check{\tilde{\eta}} \ast f ) (x,t) + \sum_{k \geq 0} U_\varphi (\check{\eta}_k \ast f ) (x,t)
$$
where $\eta_k(\xi):= \eta (2^{-k} \xi)$ and $\tilde{\eta}:= \sum_{k < 0} \eta_k$. An elementary integration by parts argument quickly reveals that the first term satisfies 
\[\| U_\varphi (\check{\tilde{\eta}} \ast f) \|_{L^q(\R^n \times [1,2])} \lesssim \| f \|_{L^p(\R^n)}\] 
for all $1 \leq p \leq q \leq \infty$. Thus, there is $(p,q,\bar \sigma)$ local smoothing for $U_\varphi$ if
\begin{equation}\label{eq:goal LpLq}
    \| U_\varphi (\check{\eta}_k \ast f ) \|_{L^q(\R^n \times [1,2])} \lesssim_{\varepsilon} 2^{k (\bar s_{p,q} - \bar \sigma + \varepsilon)} \| f \|_{L^p(\R^n)}
\end{equation}
holds for all $\varepsilon>0$ with the implicit constant uniform in $k \geq 0$.
By rescaling and setting $\lambda=2^k$, \eqref{eq:goal LpLq} is equivalent to showing
\begin{equation}\label{eq:goal LpLq scaled wave form}
    \| U_\varphi (\check{\eta} \ast f ) \|_{L^q(\R^n \times [\lambda,2\lambda])} \lesssim_{\varepsilon} \lambda^{ \beta + \varepsilon} \| f \|_{L^p(\R^n)}
\end{equation}
uniformly in $\lambda \geq 1$,
where
\[
\beta=\bar s_{p,q} - \bar \sigma + \frac{n+1}{q} - \frac{n}{p}=(n-1)\Big(\frac{1}{2}-\frac{1}{p}\Big) + \frac{1}{q} - \bar \sigma .
\]

We further note that the best constant in \eqref{eq:goal LpLq scaled wave form}
is comparable to the best constant in 
\begin{equation}\label{eq:goal LpLq scaled wave form until zero}
    \| U_\varphi (\check{\eta} \ast f ) \|_{L^q(\R^n \times [-2\lambda,2\lambda])} \lesssim_{\varepsilon} \lambda^{ \beta + \varepsilon} \| f \|_{L^p(\R^n)}
\end{equation}
uniformly in $\lambda \geq 1$. Indeed, set $\lambda = 2^k$ with $k \ge 0$.
By rescaling
\begin{multline*}
\| U_\varphi (\check{\eta}\ast f ) \|_{L^q(\R^n \times [0,2\lambda])}^{q}
    \le \sum_{j=0}^{\infty} \| U_\varphi (\check{\eta} \ast f ) \|_{L^q(\R^n \times [2^{-j}\lambda,2^{-j+1}\lambda])}^{q}\\
    \le  \sum_{j=0}^{k} \| U_\varphi (\check{\eta} \ast f ) \|_{L^q(\R^n \times [2^{-j}\lambda,2^{-j+1}\lambda])}^{q} \\
    + \sum_{j=k+1}^{\infty} 2^{-(j-k)(n+1)} \| U_\varphi (\check{\tilde{\eta}} \ast f_{j-k} ) \|_{L^q(\R^n \times [1,2])}^{q}
\end{multline*}
where
\[ \widehat{f}_j(\xi) = 2^{jn} \eta(2^{j}\xi) \widehat{f}(2^{j}\xi). \]
Note that we can add for free the Fourier localisation given by $\tilde{\eta}$.
By \eqref{eq:goal LpLq scaled wave form} 
the first term admits the desired bound by geometric summation.
By the elementary integration by parts bound $\|U_\varphi (\check{\tilde{\eta}}\ast f_{j-k}) \|_{L^q(\R^n \times [1,2])}\lesssim \| f_{j-k} \|_{L^q(\R^n)}$ and Bernstein's inequality, that is, 
\[\| f_{j-k} \|_{L^q(\R^n)} \lesssim 2^{-(j-k)n(1/p-1/q)} \| f_{j-k} \|_{L^p(\R^n)},\]
the second term admits the bound 
\[\sum_{j=k+1}^{\infty} 2^{-(j-k)} \|f \|_{L^p(\R^n )}^{q}  \]
which is also acceptable.
We conclude that the best constant in \eqref{eq:goal LpLq scaled wave form until zero}
is controlled by the best constant in \eqref{eq:goal LpLq scaled wave form}.
The other direction is immediate.

\subsection{A quantitative family of wave propagators}\label{subsec:stronger}

In \S\ref{sec:Lorentz} we will show that the class of operators $\{U_{\varphi}:  \varphi \in \Phi_{\mathrm{conic}}^+\}$ is invariant under Lorentz rescaling. To show this, it will be convenient to work with a quantitative version of the class $\Phi_{\mathrm{conic}}^+$.

Fix parameters $D_1, D_2>0, \vec{\mu}=(\mu_{\min},\mu_{\max} ) \in \R^2_+, M \geq 100n$ and $\varepsilon_\circ>0$. Let $b \in C_c^\infty(\R^n)$ be supported in 
\begin{equation*}
\Xi:=\{ \xi \in \R^n : 1/2 \leq \xi_1 \leq 2 , \,\,  |\xi_j| \leq |\xi_1| \,\, \text{for all $2 \leq j \leq n$}\}
\end{equation*}
satisfying
\begin{itemize}
    \item[B1)] $|\partial^\gamma_\xi b(\xi)| \leq D_1$ for all $\gamma \in \N_0^n$ such that $|\gamma| \leq M$.
\end{itemize}
Let $h: \R^n \to \R$ be a smooth function homogeneous of degree 1 satisfying
\begin{itemize}
    \item[H1)] $h(1,0')=\partial_\xi h (1,0')=0$;
    \item[H2)] $|\partial_{\xi}^{\gamma} h (\xi)| \leq D_2$ for all $\gamma=(\gamma_1,\gamma') \in \N_0 \times \N^{n-1}_0$ such that $|\gamma| \leq M$ and $|\gamma'| \geq 3$ and all $\xi \in \supp b$;
    \item[H3)] $|\partial^2_{\xi' \xi'} h (\xi) - \frac{1}{\xi_1} \mathrm{L} | < \varepsilon_\circ$ for some matrix $\mathrm{L} \in \mathrm{GL}(n-1,\R)$ with eigenvalues in $[\mu_{\min},\mu_{\max} ]$ and for all $\xi \in \supp b$.
\end{itemize}
It is noted that the above conditions on the derivatives imply, by homogeneity of $h$, that the remaining derivatives up to order $M$ are bounded by $C(D_2, \vec{\mu}, M,n,\varepsilon_\circ)$.  We denote by $\mathbf{H}(D_1,D_2,\vec{\mu},M,\varepsilon_\circ)$ the family of all phase-amplitude pairs $[h;b]$ satisfying B1), H1), H2) and H3), and define
$$
U_{[h;b]} f (x,t):= \frac{1}{(2\pi)^n} \int_{\R^n} e^{i ( x \cdot \xi + th (\xi))} b(\xi)\widehat{f}(\xi) \ud \xi.
$$

Given a phase $\varphi \in \Phi_{\mathrm{conic}}^+$,  the operator $U_{[\varphi;\eta]}$ in \eqref{eq:goal LpLq scaled wave form} can be written as a sum of $C(\varphi,n)$ operators of the type $U_{[h;b]}$ with $[h;b] \in \mathbf{H}(D_1,D_2,\vec{\mu},M,\varepsilon_\circ)$. 

\begin{proposition}\label{prop:main general}
Let $n \geq 2$ and $1 < p\leq q < \infty$, $s \in \R$ and $\varphi \in \Phi^+_{\mathrm{conic}}$. Assume that for any $D_1,D_2>0$, $\vec{\mu} \in \R^2_+$, $M>100n$, $\varepsilon_\circ>0$, and all $[h;b] \in \mathbf{H} (D_1,D_2,\vec{\mu},  M, \varepsilon_\circ)$, the inequality
\begin{equation*}
\| U_{[h;b]} f \|_{L^q(\R^n \times [-\lambda,\lambda])} \lesssim C_{\mathbf{H}} \lambda^{s} \| f \|_{L^p(\R^n)}
\end{equation*}
holds uniformly in $\lambda \geq 1$. Then
\begin{equation*}
\| U_{[\varphi;\eta]} f \|_{L^q(\R^n \times [-\lambda,\lambda])} \lesssim_{n, \varphi} \lambda^{s} \| f \|_{L^p(\R^n)}
\end{equation*}
holds uniformly in $\lambda \geq 1$.
\end{proposition}

\begin{proof}

Let $\varphi \in \Phi^+_{\mathrm{conic}}$. By a finite partition of unity and a rotation in the $\xi$-space, we may assume that 
\begin{equation*}
    \partial_{\xi' \xi'}^2 \varphi(1,0') \quad \text{has positive eigenvalues } \quad \text{and} \quad  \supp \widehat{f} \subseteq  [1/2,2]\times [-c_\circ,c_\circ]^{n-1} \subseteq \Xi
\end{equation*}
for some small constant $0 < c_\circ \ll 1$. This gives rise to an amplitude $b$, which satisfies the condition B$1)$ for a dimensional constant $D_1>0$ depending also on $\| \eta \|_{C^M}$, $c_\circ$, and the partition of unity. Moreover, by a translation of the $x$-space, one may add and subtract linear terms to replace the phase $\varphi$ by
\begin{equation*}
h(\xi):=\varphi(\xi)-\varphi(1,0') \xi_1 - \langle \partial_{\xi'} \varphi (1,0') , \xi' \rangle = \int_0^1 (1-r) \langle \partial^2_{\xi' \xi'} \varphi (\xi_1, r \xi') \xi', \xi' \rangle \ud r.
\end{equation*}
It then suffices to verify that $h$ satisfies H1)-H3) for some choice of $D_2,M, \varepsilon_\circ$, $\mathrm{L}$ and $\vec{\mu}$. Note that the dependency on the chosen $c_\circ$ is admissible in any case. 

To show H1), we observe that $h$ is homogeneous of degree 1 and satisfies $h(1,0')=\partial_{\xi_1} h(1,0')=\partial_{\xi'} h (1,0')=0$; note that $\varphi(1,0')=\partial_{\xi_1}\varphi(1,0')$ by homogeneity of $\varphi$. 

Regarding H2) and H3), note that $\partial^{\gamma}_\xi h(\xi)=\partial^\gamma_\xi \varphi (\xi)$ for all $\gamma  \in \N_0^n$ such that $|\gamma|\geq 2$. For fixed $M>0$, the choice $D_2=\| \varphi \|_{C^M}$ clearly verifies H2). Finally, as $\partial^2_{\xi' \xi'} h(\xi)=\partial^2_{\xi' \xi'} \varphi(\xi)$, one can take $\mathrm{L}=\partial^2_{\xi'\xi'}\varphi(1,0')$ and $\mu_{\max}$ and $\mu_{\min}$ be its largest and smallest eigenvalues. By the mean value theorem and the bounds on $|\partial^\gamma_\xi h(\xi)|$ for $\gamma \in \N_0^3$ with $|\gamma|=3$, it is clear that H3) holds with $\varepsilon_\circ=O_n(c_{\circ}D_{2})$.
\end{proof}

\subsection{Reduction to a local estimate}

For fixed time $t$, the propagators $U_{[h;b]}$ can be interpreted as Fourier multiplier operators in the $x$-variable. Thus, one may write
$U_{[h;b]}f(x,t) = K_{[h;b]}(\cdot, t) \ast f (x)$ where 
$$
K_{[h;b]}(y,t) := \frac{1}{(2\pi)^n}  \int_{\R^n} e^{i (y \cdot \xi + t h(\xi))} b (\xi) \ud \xi
$$
and $\ast$ denotes the convolution in the $x$-variable. 
As $[h;b] \in \mathbf{H}(D_1, D_2,\vec{\mu},M,\varepsilon_\circ)$, there exists $C_{\mathbf{H}}>1$ such that $|\nabla_\xi ( y \cdot \xi + t h(\xi))|  \geq |y|/2 $ for $|y| \geq C_{\mathbf{H}} \lambda$ and $|t| \leq \lambda$. 
The method of non-stationary phase hence yields
\begin{equation}
\label{eq:localreductionbound}
    |K_{[h;b]}(y,t)| \lesssim_{N,\mathbf{H}} |y|^{-N}, \quad  |y| \geq C_{\mathbf{H}}\lambda, \ |t| \leq \lambda, \ N \in \N.
\end{equation}
Denoting $\Psi_\lambda^{N}:= (1+\lambda^{-2}|\cdot |^2)^{-N/2}$,
one obtains  
\begin{equation}\label{eq:local contribution refined}
|  U_{[h;b]} f (x,t)  \bbone_{B_\lambda^{n}}(x) | \leq \big( U_{[h;b]} (f \bbone_{B_{2C_{\mathbf{H}}\lambda}^n }) (x,t) +   c_{N,\mathbf{H}} \lambda^{-N} \Psi_{\lambda}^N \ast |f| (x)  \big) \bbone_{B_\lambda^{n}}(x)
\end{equation}
for $|t| \leq \lambda$ which allows for the following local reduction.

\begin{proposition}\label{prop:local reduction}
Let $n \geq 1$, $1 < p \leq q < \infty$ and $s \in \R$. Assume that a phase amplitude pair $[h;b] \in \mathbf{H}(D_1, D_2,\vec{\mu},M,\varepsilon_\circ)$ for some fixed choice of $D_1$, $D_2>0$, $\vec{\mu} \in \R^2_+$, $M>100n$, $\varepsilon_\circ>0$ is given. Assume that
\begin{equation}\label{eq:local hyp estimate}
\| U_{[h;b]} f \|_{L^q(B_\lambda^n \times [-\lambda,\lambda])} \leq C \lambda^s \| f\|_{L^p(\R^n)} 
\end{equation}
holds uniformly in $\lambda \geq 1$ and all balls $B_\lambda^n$. Then
$$
\| U_{[h;b]} f \|_{L^q(\R^n \times [-\lambda,\lambda])} \lesssim_{p,q,n,s,\mathbf{H}} C \lambda^s \| f\|_{L^p(\R^n)}.
$$
\end{proposition}

\begin{proof}
Let $\mathcal{B}_{\lambda}^n$ be a family of finitely overlapping balls $B_\lambda^n$ covering $\R^n$. By \eqref{eq:local contribution refined} and \eqref{eq:local hyp estimate} applied to $f \bbone_{B^n_{2C_{\mathbf{H}}\lambda}}$ one has
\begin{align*}
\|U_{[h:b]} f & \|_{L^q(\R^n \times [-\lambda,\lambda])} \\
& \le \Big(  \sum_{B_\lambda^n \in \mathcal{B}_\lambda^n}     \| U_{[h;b]} f \|^q_{L^q(B_\lambda^{n} \times [-\lambda,\lambda])}  \Big)^{1/q} \\
& \leq C \lambda^{s} \Big( \sum_{B_\lambda^n \in \mathcal{B}_\lambda^n} \| f \|^q_{L^p(B_{2C_\mathbf{H}\lambda}^n )} \Big)^{1/q} + c_{N, \mathbf{H}} \lambda^{-N +1/q} \| \Psi_\lambda^{N} \ast f  \|_{L^q(\R^n)} \\
& \lesssim_{\mathbf{H}} \lambda^{s} \| f \|_{L^p(\R^n)} + c_{N, \mathbf{H}} \lambda^{-N  +1/q + n(1/q + 1/p') } \| f \|_{L^p(\R^n)} \\
& \lesssim_{p,q,n,s,\mathbf{H}} \lambda^{s} \| f \|_{L^p(\R^n)}.
\end{align*}
We used the embedding $\ell^p \subseteq \ell^q$ for $ 1 \leq p \leq q \leq \infty$ and required $N >\max\{1/q+n(1/q+1/p') - s, n\}$. 
\end{proof}

Thus, by \S\ref{subsec:dyadic} and Propositions \ref{prop:main general} and \ref{prop:local reduction}, Theorem \ref{thm:smoothing all k} follows from the following spatially and frequency localised version for the quantitative class of operators.

\begin{theorem}\label{prop:main prop LpLq}
Let $n \geq 2$ and $1 < p\leq q < \infty$ be as in Theorem \ref{thm:smoothing all k}. Let $D_1$, $D_2>0$, $\vec{\mu} \in \R^2_+$, $M>100n$, $\varepsilon_\circ>0$. Then, for all $[h;b] \in \mathbf{H} (D_1, D_2,\vec{\mu},  M, \varepsilon_\circ)$ and for any $\varepsilon>0$, the inequality
\begin{equation}\label{eq:goal LpLq scaled wave local}
\| U_{[h;b]} f \|_{L^q(B_\lambda^n \times [-\lambda,\lambda])} \lesssim_{n,p,q, \mathbf{H}, \varepsilon} \lambda^{\beta + \varepsilon} \| f \|_{L^p(\R^n)}
\end{equation}
holds uniformly in $\lambda \geq 1$ and over all balls $B_\lambda^n$, where $\beta= (n-1)\big(\frac{1}{2}-\frac{1}{p}\big) + \frac{1}{q} - \sigma_{p,q}$.
\end{theorem}

Note that \eqref{eq:goal LpLq scaled wave local} is translation invariant in the $x$-variables
so that the estimate over any ball $B_\lambda^n$ guarantees estimates over all balls $B_\lambda^n$. 

\section{$k$-broad estimates}\label{sec:kbroad}

In this section we recall the definition of the $k$-broad norm introduced in \cite{Guth16} and state the key $k$-broad estimates needed in the proof of Theorem \ref{thm:smoothing all k}.

\subsection{$k$-broad norm}\label{sec:plates}
Let $K \gg 1$ be a fixed large parameter. 
Fix a maximally $K^{-1}$-separated subset of $ \{ 1\} \times B^{n-1}(0,1)$ and for each $\omega$ belonging to this subset define the $K^{-1}$-\emph{plate}
\begin{equation*}
\tau := \big\{(\xi_1, \xi') \in \R \times \R^{n-1} : 1/2 \leq \xi_1 \leq 2 \,\, \textrm{ and } \,\, |\xi'/\xi_1 - \omega| \leq K^{-1} \big\}
\end{equation*}
and set $\omega_\tau:=\omega$.
The collection of all $K^{-1}$-plates forms a partition of $\Xi$ into finitely overlapping subsets. Consider a smooth partition of unity $\{\chi_{\tau}\}$  adapted to that covering, where $\chi_{\tau}(\xi):= \chi(K (\xi'/\xi_1 - \omega_\tau) )$ for some $\chi \in C_c^\infty(\R^{n-1})$; and set $\widehat{f}_\tau := \widehat{f} \chi_\tau$. It is also useful to consider $\tilde{\chi} \in C_c^\infty(\R^{n-1})$ such that $\tilde{\chi} \cdot \chi=\chi$ and define $\tilde{f}_\tau$ by $\widehat{\tilde{f}}_\tau:=\widehat{f}\tilde{\chi}_\tau$, where $\tilde{\chi}_\tau$ is defined analogously to $\chi_\tau$.

Let $B_{K^2}$ be a ball in $\R^{n+1}$ of radius $K^2$,  $\varphi \in \Phi^+_{\mathrm{conic}}$ and $b \in C^\infty_c(\R^n)$ supported in $\Xi$. For a fixed integer $A \geq 1$ and $1 \leq p < \infty$, define 
$$
\mu_{ U_{[\varphi;b]} f}(B_{K^2}):=\min_{V_1, \dots V_A \in \mathrm{Gr}(k-1,n+1)} \Big( \max_{\substack{\tau: \measuredangle(G(\tau), V_a) > K^{-2} \\ \mathrm{for} \:\: a=1, \dots, A}} \| U_{[\varphi;b]} f_\tau \|^p_{L^p(B_{K^2})} \Big)
$$
where 
\begin{itemize}
    \item[$\circ$] $\mathrm{Gr}(k-1,n+1)$ is the Grassmannian of all $(k-1)$-dimensional subspaces in $\R^{n+1}$;
    \item[$\circ$] $G(\tau)$ denotes the set of unit normal vectors
        \begin{equation*}
            G(\tau):= \Big\{ \frac{(-\nabla \varphi (\xi), 1)}{\sqrt{1+|\nabla \varphi(\xi)|}} : \xi \in \tau\Big\};
        \end{equation*}
    \item[$\circ$] $\measuredangle(G(\tau), V_a)$ denotes the smallest angle between non-zero vectors $v \in G(\tau)$ and $v' \in V_a$.
\end{itemize}
Let $\mathcal{B}_{K^2}$ be a collection of finitely-overlapping balls $B_{K^2}$ of radius $K^2$ which cover $\R^{n+1}$. Given any open set $W \subseteq \R^{n+1}$, define the $k$\textit{-broad norm} of $U_{[\varphi; b]}f$ over $W$ (or $k$\textit{-broad part} of $\| U_{[\varphi;b]} f \|_{L^p(W)}$) by
$$
\| U_{[\varphi;b]} f \|_{\BL_{k,A}^p(W)}:= \Big( \sum_{\substack{B_{K^2} \in \mathcal{B}_{K^2} \\ B_{K^2} \cap W \neq \varnothing}} \mu_{U_{[\varphi;b]} f} (B_{K^2}) \Big)^{1/p}.
$$
The quantity $\| U_{[\varphi;b]} f \|_{\BL_{k,A}^p(W)}$ is smaller
than the left-hand side of the conjectured multilinear estimate  \eqref{eq:multilinear}.
% It was introduced by Guth in \cite{Guth16}, to which we refer 
We refer to \cite{Guth16} and \cite[\S 6.2]{GHI} for further discussion regarding its relation with multilinear estimates.

Despite $\| U_{[\varphi;b]} f \|_{\BL_{k,A}^p}$ not being literally a norm, it satisfies versions of the triangle and Hölder's inequalities. The latter will be used in the forthcoming arguments.

\begin{lemma}[{\cite[Lemma 4.2]{Guth16}}]\label{Holder kbroad}
Let $1 \leq p, p_1,p_2 < \infty$ and $0 \leq \alpha \leq 1$ such that $\frac{1}{p}=\frac{\alpha}{p_1} + \frac{1-\alpha}{p_2}$. Suppose that $A = A_1 + A_2$ for integers $A_1,A_2 \geq 1$. Then
$$
\| U_{[\varphi;b]} f \|_{\BL_{k,A}^p(W)} \leq \| U_{[\varphi;b]} f \|_{\BL_{k,A_1}^{p_1}(W)}^{\alpha} \| U_{[\varphi;b]} f \|_{\BL_{k,A_2}^{p_2}(W)}^{1-\alpha}.
$$
\end{lemma}

\subsection{$k$-broad estimates for $U_{[\varphi;b]}$} The following $k$-broad estimates for the wave propagators $U_{[\varphi;b]}$ are a key ingredient in the proof of Theorem \ref{thm:smoothing all k}. 

\begin{theorem}[{\cite[Theorem 5.3]{GLMX}}, {\cite[Theorem 1.2]{Schippa}}]\label{thm:broad estimate}
Let $D_1$, $D_2>0$, $\vec{\mu} \in \R^2_+$, $M>100n$, $\varepsilon_\circ>0$.
For any $2 \leq k \leq n+1$ and any $\varepsilon>0$, there is a large integer  $1 \ll A \lesssim K^\varepsilon$ and $d_\varepsilon>1$
so that
$$
\| U_{[h;b]} f \|_{\BL_{k,A}^p(B_\lambda)} \lesssim_{\varepsilon, \mathbf{H}} K^{d_\varepsilon} \lambda^\varepsilon \| f \|_{L^2(\R^n)}
$$
holds for any $1\leq K^{\varepsilon} \lesssim \lambda^{\varepsilon^2}$, any $[h; b] \in \mathbf{H}(D_1,D_2,\vec{\mu},M,\varepsilon_\circ)$ and any $p \geq \bar{p}_{n,k}:=\frac{2(n+k+1)}{n+k-1}$ uniformly over all balls $B_\lambda$ of radius $\lambda$.
\end{theorem}
Note that the parameter $A$ can be chosen independently of the location of the ball $B_\lambda$, as a translation of the ball only induces an admissible modulation on $\widehat{f}$.

\subsection{Reverse Hölder inequality: a decomposition lemma}

We are not aware of interpolation theorems that would apply to $k$-broad estimates,
but this problem can be circumvented by means of an additional decomposition of the input data if the desired $k$-broad estimate allows for $\varepsilon$-losses.
Namely, for functions satisfying a reverse Hölder type inequality, 
it is straightforward to interpolate $k$-broad estimates at the expense of increasing the parameter $A$,
and further, 
bounded $L^p$ functions can be decomposed, up to an error term, 
into finitely many pieces satisfying a reverse Hölder inequality. This last fact can be seen as a finite version of the decomposition used in Marcinkiewicz interpolation theorem.

\begin{lemma}\label{lemma:interpolation kbroad}
Let $1 \leq p, p_1,p_2, q, q_1, q_2 < \infty$ and $0 \leq \alpha \leq 1$ such that $\frac{1}{p}=\frac{\alpha}{p_1} + \frac{1-\alpha}{p_2}$ and $\frac{1}{q}=\frac{\alpha}{q_1} + \frac{1-\alpha}{q_2}$. Suppose that $A=A_1+A_2$ for integers $A_1,A_2 \geq 1$. Assume that
\begin{equation}\label{eq:interpolation hyp}
\| U_{[\varphi;b]} f \|_{\BL_{k,A_i}^{q_i}(W)} \leq C_i \| f \|_{p_i} \qquad \text{for } i=1,2.
\end{equation}
Then
$$
\| U_{[\varphi;b]} f \|_{\BL_{k,A}^{q}(W)} \leq C C_1^\alpha C_2^{1-\alpha}  \| f \|_{p}
$$
for all functions $f$ satisfying the reverse Hölder inequality 
\begin{equation}\label{eq:reverse Holder hyp}
\| f \|_{p_1}^{\alpha} \| f \|_{p_2}^{1-\alpha} \leq C \| f \|_p,
\end{equation}
\end{lemma}

\begin{proof}
This follows from Lemma \ref{Holder kbroad} and the hypotheses \eqref{eq:interpolation hyp} and \eqref{eq:reverse Holder hyp}.
\end{proof}

The relevant decomposition can be found,
for instance,
in the informal lecture notes \cite{Guth_lecture_notes},
but we recall the proof below for completeness.
% The assumption on compact Fourier support is only needed 
% to conclude the last inequality of item i).
% It could hence be omitted at the cost of replacing the $L^p$ norm by an $L^\infty$ norm,
% but we prefer to state the lemma in the ad hoc form in which it will be applied later.

\begin{lemma}\label{lemma:decomposition RH}
Let $1 \leq p < \infty$ and fix $R \ge 1$ and $m > 0$.
Assume that $f \in L^p(\R^n) \cap L^\infty(\R^n)$.
Then $f$ can be written as
\begin{equation}\label{eq:decomposition}
f= \sum_{\nu=0}^{ m \lfloor\log R \rfloor} f^\nu + e,
\end{equation}
where
\begin{enumerate}[i)]
    \item $\| e \|_{L^\infty(\R^n)} \lesssim  R^{-m} \| f \|_{L^\infty(\R^n)}$;
    \item $\| f^\nu \|_{L^r(\R^n)} \lesssim \| f \|_{L^r(\R^n)}$ for any $1 \leq r \leq \infty$ and all $\nu=0,\dots, m \lfloor \log R \rfloor$;
    \item if $1 \leq r, r_1, r_2 \leq \infty$ satisfy $\frac{1}{r}= \frac{\alpha}{r_1} + \frac{1-\alpha}{r_2}$ for some $0 \leq \alpha \leq 1$, then
\begin{equation*}
\| f^\nu \|_{L^{r_1}(\R^n)}^{\alpha} \| f^{\nu} \|_{L^{r_2}(\R^n)}^{1-\alpha} \leq 2 \| f^\nu \|_{L^r(\R^n)}
\end{equation*}
for all $\nu=0,\dots, m\lfloor \log R \rfloor$.
\end{enumerate}
\end{lemma}

\begin{proof}
Let $\nu_{\circ} \in \Z$ such that $2^{\nu_\circ - 1} < \| f \|_\infty \leq 2^{\nu_\circ}$. For any $\nu \in \Z$, let $f^\nu := f \bbone_{ \{2^{\nu -1} < |f| \leq 2^{\nu} \}}$ and write  $f= \sum_{\nu=- \infty}^{\nu_{\circ}} f^\nu$. Then ii) follows by definition. 
Because $1 \le p < \infty$ and 
\begin{equation*}
    2^{\nu-1} |\supp f^\nu|^{1/p} \leq \| f^\nu \|_p \leq \| f \|_p < \infty,
\end{equation*} 
we have that $|\supp f^\nu|< \infty$.
This immediately implies iii). Furthermore, writing $f= \sum_{\nu=\nu_{\circ} - m \lfloor  \log R \rfloor }^{\nu_\circ} f^\nu + e$ for any fixed $R \geq 1$ and $m>0$, one has that $\| e \|_{\infty} \lesssim R^{-m} \| f \|_\infty $. This implies i) and \eqref{eq:decomposition} follows by relabelling $\nu$.
\end{proof}

\subsection{Local smoothing estimate for $k$-broad norms} 
Any local smoothing estimate implies a corresponding $k$-broad estimate.

\begin{proposition}\label{cor:k-broad smoothing}
Let $n \geq 1$, $2 \leq k \leq n+1$, $K \geq 2$ and $A \geq 1$. Let $\bar p_n \geq \frac{2n}{n-1}$ and assume that the $(\overline{p}_n,\overline{p}_n,1/\overline{p}_n)$ local smoothing estimate holds. Then for any $\varepsilon>0$, the inequality
$$
\| U_{[\varphi;b]} f \|_{\BL_{k,A}^{\bar p_n}(B_R^n \times [-R,R])} \lesssim_\varepsilon   R^\varepsilon R^{ (n-1)(\frac{1}{2}-\frac{1}{\bar p_n}) } \| f \|_{L^{\bar p_n}(\R^n)}
$$
also holds for any $R \geq 1$, with constant independent of $A$.
\end{proposition}

\begin{proof}
By definition of the $k$-broad norm and the embedding $\ell^{\bar p_n} \subseteq \ell^\infty$
\[
\| U_{[\varphi;b]} f \|_{\BL_{k,A}^{\bar p_n}(B_R^n \times [-R,R])} \leq \Big( \sum_{B_{K^2} \subset B_R^n \times [-R,R]} \sum_{\tau} \| U_{[\varphi;b]} f_\tau \|_{L^{ \bar p_n}(B_{K^2})}^{\bar p_n} \Big)^{1/ \bar p_n}  ,
\]
where the sum in $\tau$ ranges over all $K^{-1}-$plates. The claim follows by changing the order of summation, applying the hypothetical local smoothing estimate on $\| U_{[\varphi;b]}   f_\tau  \|_{L^{\bar  p_n}(B_{R}^n \times [-R,R])}$ (via \S\ref{subsec:dyadic}) and using the bound $(\sum_\tau \| f_\tau \|_{\bar p_n}^{\bar p_n})^{1/\bar p_n} \lesssim \| f \|_{\bar p_n}$, which follows by interpolation between the cases $p=2$ and $p=\infty$.
\end{proof}

\section{Narrow decoupling and flat phases}\label{subsec:narrow decoupling}
If the contribution to $U_{[\varphi;b]} f$ comes from plates whose normal vectors lie close to a $(k-1)$-dimensional subspace, one can essentially use the Bourgain--Demeter decoupling inequality \cite{BD2015} in $\R^{k-1}$. This phenomenon is normally referred to as \textit{narrow} decoupling. The case $\phi(\xi)=|\xi|$ was established by Harris \cite[Theorem 2.3]{Harris18} and can be used to show that the same result holds for suitably small perturbations of $|\xi|$.

\begin{definition}\label{def:flat}
Let $D_1,D_2>0$, $\vec{\mu}\in \R^2_+$, $M>100n$, $\varepsilon_\circ>0$. 
Let $L > 0$. Given a phase-amplitude pair $[h;b] \in \mathbf{H} (D_1,D_2, \vec{\mu}, M, \varepsilon_\circ)$, we say that $[h;b]$
is $L$-flat if
\begin{equation*}
    |\partial_{\xi}^\alpha h(\xi)| \lesssim L^{-1} D_2 \quad \text{ for $|\alpha'| \geq 3$, $\,|\alpha|\leq M$, $\, \xi \in \supp b$,}
\end{equation*}
where $\alpha=(\alpha_1,\alpha') \in \N_0 \times \N_0^{n-1}$.
\end{definition}

If $[h;b] \in \mathbf{H}(D_1,D_2,\vec{\mu},M,\varepsilon_\circ)$ is $L$-flat, Taylor expansion immediately reveals that
\begin{equation*}
    h(\xi_1,\xi')=\frac{\langle \partial_{\xi'\xi'}^2h(1,0')\xi',\xi' \rangle}{2 \xi_1} + L^{-1} E(\xi),
\end{equation*}
where $E$ is homogeneous of degree 1 and $|\partial^\alpha E(\xi)| \lesssim 1$ for all $|\alpha| \leq M-3$ on $\supp b$. This concept was introduced in \cite{GLMX} (see also \cite{Schippa}) to deduce the following narrow decoupling inequality.

\begin{theorem}\label{thm:narrow decoupling}
Let $n \geq 2$, $3 \leq k \leq n+1$ and $K \geq 2$. Let $D_1$, $D_2>0$, $\vec{\mu} \in \R^2_+$, $M>100n$, $\varepsilon_\circ>0$. Let $[h;b] \in \mathbf{H}(D_1,D_2,\vec{\mu},M,\varepsilon_\circ)$ be $K^2$-flat. Then for any $\varepsilon>0$ and $N>0$, the inequality
$$
\| U_{[h; b]} f \|_{L^p(B_{K^2}) } \lesssim_{\varepsilon, \mathbf{H},N} K^\varepsilon \Big( \sum_{\tau} \| U_{[h;b]}f_\tau \|^2_{L^p(w_{B_{K^2}}^{N})} \Big)^{1/2}
$$
holds for all $2 \leq p \leq \frac{2(k-1)}{k-3}$ whenever $U_{[h;b]}f= \sum_{\tau} U_{[h;b]} f_\tau$ and $\tau$ are $K^{-1}$-plates such that $\measuredangle (G(\tau), V) \leq K^{-2}$ for some $(k-1)-$dimensional vector space $V$.
\end{theorem}

In order to see this, let $[h;b] \in \mathbf{H}(D_1,D_2,\vec{\mu},M,\varepsilon_\circ)$ be $K^2$-flat and let $\tilde{h}$ denote its second order Taylor polynomial. By a suitable change of variables, the result of Harris for $\phi(\xi)=|\xi|$ can be extended to the phase $\tilde{h}$. For the extension to $h$, let $\Gamma_h^{K}$ denote the $K^{-2}$ neighbourhood of the cone generated by $h$ and $\Gamma_{\tilde{h}}^{K}$ its analogue for $\tilde{h}$. Because of the $K^{2}$-flat hypothesis, the objects $\Gamma_{h}^{K}$ and $\Gamma_{\tilde{h}}^{K}$ are indistinguishable from one another and, moreover, the normals to $\Gamma_h^{K}$ lie in the $K^{-2}$-neighbourhood of the normals to $\Gamma_{\tilde{h}}^{K}$. Hence the decoupling inequality extends to the $K^{2}$-flat case (via its equivalent formulation in terms of the Fourier support lying on a neighbourhood of a cone).

\section{Lorentz rescaling}\label{sec:Lorentz}

\subsection{Lorentz rescaling} 
We will next prove that the target estimate \eqref{eq:goal LpLq scaled wave local} self-improves if the support of $\widehat{f}$ is small. This is achieved by applying a standard Lorentz rescaling argument.

Before turning to the proof, it is instructive to compare the situation with Fourier restriction estimates, which are of the type \eqref{eq:goal LpLq scaled wave local} but with the right-hand side replaced by $\| \widehat{f} \|_p$. Such estimates are invariant under rotations in $\R^{n+1}$, and one can then apply the rotation $L(\xi_1, \xi', \tau )= (\xi_{1} + \tau, \xi', \tau - \xi_1 )$, where $(\xi_1, \xi', \tau) \in \R \times \R^{n-1} \times \R$, which maps the forward light cone $ \Gamma := \{ (\xi, \tau) \in \R^{n} \times \R : \tau= |\xi|\}$ into the \textit{tilted} cone $ \Gamma_{\mathrm{par}}:=\{ (\xi_1, \xi', \tau) \in \R \times \R^{n-1} \times \R : \tau = |\xi'|^2/\xi_1\}$. Thus, Fourier restriction estimates for the phase $\varphi(\xi)=|\xi|$ follow from those for $h_{\mathrm{par}}(\xi)=|\xi'|^2/\xi_1$. The phase function $h_{\mathrm{par}}$ satisfies the special property of being invariant under Lorentz rescaling, due to its perfect parabolic structure.

The invariance under Lorentz rescaling is no longer true for local smoothing estimates for $e^{i t \sqrt{-\Delta}}$, as they are not rotationally invariant in $\R^{n+1}$. However, the class of phase functions in $\mathbf{H}(D_1, D_2,\vec{\mu},M,\varepsilon_\circ)$ is invariant under rescaling: given a generic $h$ in this class, the rescaled phase $\tilde{h}$ is different from  the original $h$, but  still satisfies H1), H2) and H3). This is the underlying reason for introducing the larger family of wave-propagators $U_\varphi$ when proving estimates for $e^{i t \sqrt{-\Delta}}$ via an induction-on-scales argument.

\begin{lemma}\label{lemma:Lorentz}
Let $n \geq 2$ and $1 < p\leq q < \infty$ be as in Theorem \ref{thm:smoothing all k}. Let $D_1$, $D_2>0$, $\vec{\mu}\in \R^2_+$, $M>100n$, $\varepsilon_\circ>0$ and $L>0$. Assume \eqref{eq:goal LpLq scaled wave local} holds for all $[h;b] \in \mathbf{H} (D_1, D_2,\vec{\mu}, M, \varepsilon_\circ)$ that are $L$-flat and all $\lambda \geq 1$. Let $K \geq 2$ be sufficiently large, depending on $n$ and $M$, and $\tau$ be a $K^{-1}$-plate. Then
$$
\| U_{[h;b]} f_\tau \|_{L^q(B_R^n \times [-R,R])} \lesssim_{\mathbf{H}, \varepsilon}  K^{\frac{n+1}{q} - \frac{n-1}{p} } (R/K^2)^{\beta + \varepsilon} \| \tilde{f}_\tau \|_{L^p(\R^n)},
$$
where $f_\tau$ and $\tilde{f}_\tau$ are defined as in \S\ref{sec:plates}.
\end{lemma}

\begin{proof}
Let $(1,\omega) \equiv (1, \omega_\tau)$ be the center of the $K^{-1}$-plate $\tau$ upon which $\widehat{f}_\tau$ is supported. Perform the change of variables $(\xi_1,\xi')=(\eta_1, \eta_1 \omega + K^{-1} \eta')$, so that $h(\xi)=h(\eta_1, \eta_1 \omega + K^{-1} \eta')$. By a Taylor expansion around $(\eta_1, \eta_1 \omega)$ and the homogeneity of $h$, $h(\xi)$ equals to
\begin{align}\label{eq:Taylor}
\eta_1 h(1, \omega) + K^{-1} \langle \partial_{\xi'} h (1,\omega) ,  \eta' \rangle + K^{-2}\int_0^1 (1-r) \langle  \partial^2_{\xi' \xi'} h (1,  \omega + r K^{-1}\frac{ \eta'}{\eta_1})  \eta',  \eta'  \rangle  \frac{\ud r}{\eta_1}.
\end{align}
Let $\tilde{h}(\eta)$ be the function associated with the integral above,
\begin{equation}\label{eq: h tilde}
\tilde{h}(\eta)= K^2 h(\eta_1, \eta_1 \omega + K^{-1} \eta') - K^2 \eta_1 h(1, \omega) - K \langle \partial_{\xi'} h (1,\omega),  \eta' \rangle.
\end{equation}
Let $\mathrm{D}_{K}, \Upsilon_\omega: \R \times \R^{n-1} \times \R \to \R^{n+1}$ be linear functions given by
\begin{align*}
\mathrm{D}_{K}(x_1,x',t) & =(x_1, K^{-1}x', K^{-2}t) \\
\Upsilon_\omega(x_1,x',t) & =(x_1 + \langle x', \omega \rangle  + t h (1,\omega),x'+t \partial_{\xi'} h(1,\omega),t).
\end{align*}
Note that
\begin{equation*}
U_{[h;b]} f_{\tau}(x,t) = U_{[\tilde{h};\tilde{b}]}  g (\mathrm{D}_K \circ \Upsilon_\omega (x_1,x',t))
\end{equation*}
where\footnote{Technically, one should divide $\tilde{b}$ and multiply $g$ by a dimensional constant to ensure that $\tilde{b}$ satisfies B1). This only causes an admissible dimensional constant loss in the resulting inequality.}
\begin{align}
    \widehat{g}(\eta)&:=\widehat{\tilde{f}_\tau}(\eta_1, \eta_1 \omega + K^{-1} \eta')  K^{-{(n-1)}} \notag \\
    \tilde{b}(\eta)&:=b(\eta_1,\eta_1\omega +  K^{-1} \eta')\chi(\eta') \label{eq:b tilde def}.
\end{align}
We then have
\begin{equation}\label{eq:rescaling op}
\| U_{[h;b]} f_\tau \|_{L^q(B_R^n \times [-R,R])} = K^{\frac{n+1}{q}} \| U_{[\tilde{h};\tilde{b} ]} g  \|_{L^q( \mathrm{D}_K \circ \Upsilon_\omega (B_R^n \times [-R,R]))}
\end{equation}
and
\begin{equation}\label{eq:rescaling funct}
\| g \|_{L^p(\R^n)} = K^{- \frac{n-1}{p}}\| \tilde{f}_\tau \|_{L^p(\R^n)}.
\end{equation}
Let $\mathcal{B}_{R/K^2}$ be a finitely overlapping collection of cylinders of the form
\[
B_{R/K^2}\equiv B_{R/K^2}^n \times [-R/K^2, R/K^2]
\]
such that
$$  \mathrm{D}_K \circ \Upsilon_\omega (B_R^n \times [-R,R]) \subseteq \bigcup_{B_{R/K^2} \in \mathcal{B}_{R/K^2}} B_{R/K^2}.$$
Assuming temporarily that $[\tilde{h};\tilde{b}]$ belongs to $\mathbf{H}(D_1, D_2,\vec{\mu},M, \varepsilon_\circ)$ and is $L$-flat, 
we may use the hypothesis \eqref{eq:goal LpLq scaled wave local} on each $B_{R/K^2}$ to deduce
$$
\| U_{[\tilde{h};\tilde{b}]} g \|_{L^q(B_{R/K^2})} \lesssim_{n,p,q,\mathbf{H}, \varepsilon} (R/K^2)^{\beta+\varepsilon} \| g \|_{L^p(\R^n)}.
$$
By Proposition \ref{prop:local reduction} (at scale $R/K^2$),
this implies
$$
\| U_{[\tilde{h};\tilde{b}]} g \|_{L^q(\mathrm{D}_K \circ \Upsilon_\omega (B_R^n \times [-R,R]))} \lesssim_{n,p,q,\mathbf{H}, \varepsilon} (R/K^2)^{\beta+\varepsilon} \| g \|_{L^p(\R^n)}.
$$
Combining this with \eqref{eq:rescaling op} and \eqref{eq:rescaling funct} allows us to conclude
$$
 \| U_{[h;b]} f_\tau \|_{L^q(B_R^n \times [-R,R])} \lesssim K^{\frac{n+1}{q} - \frac{n-1}{p} } (R/K^2)^{\beta + \varepsilon } \| f_\tau \|_{L^p(\R^n)}.
$$

It remains to verify that $[\tilde{h};\tilde{b}]  \in \mathbf{H}(D_1, D_2,\vec{\mu},M, \varepsilon_\circ)$ and that is $L$-flat. It follows from the expression of $\tilde{b}$ in \eqref{eq:b tilde def} that it is supported in $\Xi$ and satisfies B1) (see footnote 2). Regarding the phase $\tilde{h}$, it is clear from its definition in \eqref{eq: h tilde} that it is homogeneous of degree 1. Moreover, either \eqref{eq: h tilde} and the homogeneity of $h$, or simply the integral expression \eqref{eq:Taylor} quickly reveal that 
\[\tilde{h}(1,0')=\partial_{\eta_1} \tilde{h} (1,0')=\partial_{\eta'}\tilde{h}(1,0')=0.  \]
This verifies that H1) holds.
Furthermore, note that \eqref{eq: h tilde} yields
$$
\partial_{\eta'}^{\gamma'} \tilde{h} (\eta) = K^{-(|\gamma'|-2)} \partial_{\xi'}^{\gamma'} h (\eta_1, \eta_1 \omega + K^{-1} \eta')
$$
for any $\gamma' \in \N_0^{n-1}$ such that $|\gamma'| \geq 2$. This and the assumptions on $h$ immediately imply H2) and H3) for $\tilde{h}$, provided that $K \geq 2$ is sufficiently large depending on $M$ and $n$. Moreover, as $[h;b]$ is $L$-flat, the above identity also implies that $[\tilde{h};\tilde{b}]$ is $L$-flat.
\end{proof}

\begin{remarkn}\label{rem:flat}
We emphasise that if $[h;b]$ is $L$-flat, then the rescaled pair $[\tilde{h};\tilde{b}]$ is $LK$-flat,
as can be read from the proof above.
This fact will be referred to later.
\end{remarkn}

\section{Proof of Theorems \ref{thm:LpLq} and \ref{thm:smoothing all k}} \label{sec:proof}

As discussed in Section \ref{sec:initial}, Theorem \ref{thm:smoothing all k} is a consequence of Theorem \ref{prop:main prop LpLq}, which can be further reduced to an equivalent statement for flat functions.
Indeed, fix\footnote{If $\lambda \lesssim 1$,  Theorem \ref{thm:smoothing all k} holds from the kernel estimate \eqref{eq:localreductionbound}.} $\lambda \gg 1$, $\varepsilon>0$, and a pair $[h;b] \in \mathbf{H}(D_1, D_2,\vec{\mu},M, \varepsilon_\circ)$. Let $ \tilde{\delta}=  \frac{\varepsilon}{2(n-1)}>0$ and decompose the support of $b$ into $\lambda^{-\tilde{\delta}}$-plates. Applying the Lorentz rescaling Lemma \ref{lemma:Lorentz} to each piece, the rescaled phase-amplitude pairs are in $\mathbf{H}(D_1,D_2,\vec{\mu},M,\varepsilon_\circ)$ and are $\lambda^{\tilde{\delta}}$-flat (see Remark \ref{rem:flat}). As there are $O(\lambda^{\tilde{\delta}(n-1)})$ many plates, \eqref{eq:goal LpLq scaled wave local} follows if we can prove that for all $[h;b] \in \mathbf{H}(D_1,D_2,\vec{\mu},M,\varepsilon_\circ)$ that are $\lambda^{\tilde{\delta}}$-flat, the inequality
\begin{equation}\label{eq:flat reduction}
\| U_{[h;b]} f \|_{L^q(B_\lambda^n \times [-\lambda,\lambda])} \lesssim_{n,p,q, \mathbf{H}, \varepsilon} \lambda^{\beta + \varepsilon/2} \| f \|_{L^p(\R^n)}
\end{equation}
holds uniformly in $\lambda \geq 1$ 
and over all balls $B_\lambda^n$. Here again 
\[\beta= (n-1)\Big(\frac{1}{2}-\frac{1}{p}\Big) + \frac{1}{q} - \sigma_{p,q}.\]
To this end, we introduce the following definition.
\begin{definition}
\label{def:qquantity}
Given $\varepsilon>0$, $R \geq 1$, $1 < p \leq q < \infty$ and $\beta= (n-1)\big(\frac{1}{2}-\frac{1}{p}\big) + \frac{1}{q} - \sigma_{p,q}$, let $Q_{\varepsilon,p,q} (R)$ denote the infimum over all constants $C \geq 0$ such that the inequality
$$
\| U_{[h;b]} f \|_{L^q(B_R)} \leq C R^{\beta + \varepsilon} \| f \|_{L^p(\R^n)}
$$
holds for all cylinders $B_R = B_R^n \times [-R,R]$, all phase/amplitude pairs $[h;b] \in \mathbf{H}(D_1, D_2,\vec{\mu},M, \varepsilon_\circ)$ that are $\lambda^{\varepsilon/(n-1)}$-flat, all $\lambda \geq R$ and all functions $f \in L^p(\R^n)$.
\end{definition}

Thus, in order to verify \eqref{eq:flat reduction}, it suffices to show that for any $\varepsilon>0$,
\begin{equation}\label{eq:induction quantity}
Q_{\varepsilon, p , q}(R) \le C(\varepsilon)
\end{equation}
for all $R\geq 1$. 
The constant $C(\varepsilon)$ is allowed to depend on the quantities listed 
in Definition \ref{def:qquantity}, namely $p$, $q$, $n$, $D_1$, $D_2$, $\vec{\mu} \in \R^2_+$, $M$ and $\varepsilon_\circ$.
We do not track dependencies on them from this point on and, 
whenever necessary,
we refer to them as the data.
The proof of \eqref{eq:induction quantity} will proceed via induction on scales. 

By the kernel estimate \eqref{eq:localreductionbound}, the inequality \eqref{eq:induction quantity} holds for small values $R \lesssim_\varepsilon 1$. This allows us to induct on the quantity $R$, with $R \lesssim_\varepsilon 1$ as a base case. In particular, one can state the following induction hypothesis.
\begin{hypothesis}
There exists a constant $\bar C_\varepsilon$ depending only on $\varepsilon$ and the data such that
$$
Q_{\varepsilon, p, q}(R') \leq \bar C_\varepsilon
$$
holds for all $1 \leq R' \leq R/2$.
\end{hypothesis} 

We shall next show that $Q_{\varepsilon, p, q}(R) \leq \bar C_\varepsilon$. Let $f \in L^p(\R^n)$. By the support properties of $b$, we can assume that $\supp \widehat{f}  \subseteq B(0,10)$, which by Young's convolution inequality implies $f \in L^\infty(\R^n)$. Thus, by Lemma \ref{lemma:decomposition RH} and the triangle inequality one has
\begin{equation}\label{eq:norm balanced dec}
    \| U_{[h;b]} f \|_{L^q(B_R)} \leq \sum_{\nu=0}^{m \lfloor \log R \rfloor} \| U_{[h;b]}  f^\nu  \|_{L^q(B_R)} + \| U_{[h;b]} e \|_{L^q(B_R)}
\end{equation}
for any $m>0$. By Hölder's inequality and the kernel estimate \eqref{eq:localreductionbound}, one has
$$
\| U_{[h;b]} e \|_{L^q(B_R)} \lesssim R^{\frac{n+1}{q}} \| \Psi_{R}^{n+1} \ast |e| \|_{L^\infty (B_R^n)} \lesssim R^{\frac{n+1}{q} + n} \| e \|_{L^{\infty}(\R^n)}.
$$
Using that $\| e \|_{L^\infty(\R^n)} \lesssim R^{-m} \| f \|_{L^\infty(\R^n)} \lesssim R^{-m} \| f \|_{L^p(\R^n)}$, one readily obtains
\begin{equation}\label{eq:error balanced}
\| U_{[h;b]} e \|_{L^q(B_R)} \lesssim  \| f \|_{L^p(\R^n)}
\end{equation}
provided $m>\frac{n+1}{q} + n$.

\subsection*{Broad-narrow analysis} We shall next perform a Bourgain--Guth broad--narrow analysis (cf.\ \cite{BG,Guth16}) on each function $f^\nu$.
By Theorem \ref{thm:broad estimate}, there exists an integer $1 \ll A \lesssim_\varepsilon K^{\varepsilon/2}$ independent of the ball $B_R$ such that
$$
\| U_{[h;b]} f^\nu \|_{\BL_{k,A}^{\bar p_{n,k}}(B_R)} \lesssim_{\varepsilon} K^{d_\varepsilon} R^{\varepsilon /2} \| f^\nu \|_{L^2(\R^n)}
$$
provided $K^{\varepsilon/2} \lesssim R^{\varepsilon^2/4}$. Here
$K \geq 2$ is the parameter used to define the $k$-broad norm and will be specified later. Moreover, by Proposition \ref{cor:k-broad smoothing}
$$
\| U_{[h;b]} f^\nu \|_{\BL_{k,1}^{\bar p_n}(B_R)} \lesssim_{ \varepsilon} R^{\varepsilon/2} R^{(n-1)\left( \frac{1}{2} - \frac{1}{\bar p_n} \right) } \| f^\nu \|_{L^{\bar p_n}(\R^n)},
$$
where $\bar p_n$ is the (hypothetical) smallest exponent 
for which sharp local smoothing holds.

By iii) of Lemma \ref{lemma:decomposition RH}, $f^\nu$ satisfies the reverse Hölder inequality \eqref{eq:reverse Holder hyp}.
Hence one can interpolate the above inequalities via Lemma \ref{lemma:interpolation kbroad} and use ii) of Lemma \ref{lemma:decomposition RH} to obtain
\begin{equation}\label{eq:broad main}
\| U_{[h;b]} f^\nu \|_{\BL_{k,A+1}^q (B_R)} \lesssim_{ \varepsilon} K^{d_\varepsilon} R^{\varepsilon/2} R^{(n-1)\left( \frac{1}{2} - \frac{1}{p} \right) } \| f \|_{L^p(\R^n)}
\end{equation}
for 
\[\frac{1}{p}-\frac{1}{\bar p_n}=(n+k+1) \Big( \frac{1}{2} - \frac{1}{\bar p_n} \Big) \Big(\frac{1}{p} - \frac{1}{q} \Big),\] 
$2 \leq p \leq \bar p_n$ and $\bar p_{n,k} \leq q \leq \bar p_n$. 

Consider next the decomposition of the support of $b$ into $K^{-1}$-plates $\tau$.
For each $B_{K^2}\subset B_R$, let $V_1,\dots, V_{A+1}$ be a collection of $(k-1)$-dimensional subspaces in $\R^{n+1}$ attaining the minimum
$$
\min_{V_1, \dots V_{A+1} \in \mathrm{Gr}(k-1,n+1)} \Big( \max_{\tau \not \in V_a } \| U_{[h;b]}   f_\tau^\nu \|^q_{L^q(B_{K^2})} \Big),
$$
where $\tau \not \in V_a$ stands for $\measuredangle(G(\tau), V_a) > K^{-2}$ for all $a=1, \dots, A+1$. 
Then
$$
\int_{B_{K^2}} |U_{[h;b]} f^\nu |^q \lesssim K^{C} \max_{\tau \not \in V_a} \int_{B_{K^2}} |U_{[h;b]}  f_\tau^\nu |^q + \sum_{a=1}^{A+1} \int_{B_{K^2}} \big| \sum_{\tau \in V_a} U_{[h;b]} f_\tau^\nu \big|^q.
$$
When summing over $B_{K^2} \subset B_R$, the first term corresponds to the \textit{broad} part $\| U_{[h;b]} f^\nu \|_{\BL_{k,A+1}^{q}(B_R)}^q$, which satisfies the estimate \eqref{eq:broad main}. The second term corresponds to the \textit{narrow} part, for which the plates accumulate on a $(k-1)$-dimensional subspace. Provided that $K^2 \leq \lambda^{\varepsilon/(n-1)}$, the pair $[h;b]$ is $K^2$-flat. By Theorem \ref{thm:narrow decoupling} and Hölder's inequality, for every $\delta'>0$ and $N>0$,
$$
\int_{B_{K^2}} |\sum_{\tau \in V_a} U_{[h;b]}  f_\tau^\nu|^q \lesssim_{\delta',N} K^{q \delta'} \max \{1,K^{q(k-3)(\frac{1}{2} - \frac{1}{q} )}\} \sum_{\tau \in V_a} \int_{\R^{n+1}} |U_{[h;b]}   f_\tau^\nu|^q w_{B_{K^2}}^N
$$
holds for all $2\leq q \leq \frac{2(k-1)}{k-3}$ (with $2 \leq q \leq \infty$ for $k \in \{2,3\}$), using Hölder's inequality in the sum and noting that there are $O(K^{k-3})$ $K^{-1}$-plates $\tau \in V_a$. As we have already taken advantage of the reduced number of plates, we can further control the sum over $\tau \in V_a$ by the sum over all $ K^{-1}$-plates $\tau$. Thus, summing over $a$ and the balls $B_{K^2} \subset B_R$,
\begin{align*}
     \Big ( \sum_{B_{K^2} \subset B_{R}} & \sum_{a=1}^{A+1} \int_{B_{K^2}} |\sum_{\tau \in V_a} U_{[h;b]}  f_\tau^\nu|^q \Big)^{1/q} \\
     & \lesssim_{\delta',N} A^{1/q} K^{\delta'} \max \{1, K^{(k-3)(\frac{1}{2} - \frac{1}{q} )}\} \Big( \sum_{\tau: K^{-1}-\mathrm{plates}} \|U_{[h;b]}  f_\tau^\nu \|_{L^q( w_{B_R}^{N})}^q \Big)^{1/q},
\end{align*}
where we used $\sum_{B_{K^2} \subset B_R} w_{B_{K^2}}^N \lesssim w_{B_R}^N$. 

Next, note that for any $\delta>0$ and $\tilde{N}>0$ one has
$$
\| U_{[h;b]}   f_\tau^\nu \|_{L^q( w_{B_R}^N)} \lesssim_{\delta,\tilde{N}}  \| U_{[h;b]}   f_\tau^\nu \|_{L^q(R^\delta B_{ R})} + R^{-\tilde{N}} \| f_\tau^\nu \|_{L^p(\R^n)}.
$$
This follows from the kernel estimate \eqref{eq:localreductionbound} and the decay of the weight $w_{B_R}^N$ on $\R^{n+1}\backslash R^\delta B_R$ provided $N$ is chosen sufficiently large depending on $\delta, \tilde{N}, p, q$ and $n$.
Applying a trivial decoupling (via the triangle inequality) of $K^{-1}$-plates $\tau$ into $(KR^{\delta/2})^{-1}$-plates $\tau'$, we obtain
\begin{multline}
     \Big ( \sum_{B_{K^2} \subset B_{R}}  \sum_{a=1}^{A+1} \int_{B_{K^2}} |\sum_{\tau \in V_a} U_{[h;b]}  f_\tau^\nu|^q \Big)^{1/q}  \\
     \lesssim %_{\delta,\delta',\tilde{N}} A^{1/q} 
     R^{\delta\frac{(n-1)}{2}} K^{\delta'+\varepsilon} \max \{1, K^{(k-3)(\frac{1}{2} - \frac{1}{q} )}\} 
      \Big( \sum_\tau \sum_{\tau'\cap \tau \neq \varnothing } \|U_{[h;b]}   (f_{\tau}^\nu)_{\tau'}  \|_{L^q(  R^\delta B_R)}^q \Big)^{1/q} \\
     +R^{-\tilde{N}} \Big( \sum_{\tau} \|f_{\tau}^\nu \|_{L^p( \R^n )}^q \Big)^{\frac{1}{q}} \label{eq:narrow triv dec}
\end{multline}
where we have used $A \lesssim K^{\varepsilon/2}$ and the constant depends on $\delta,\delta',\tilde{N}$.
Using the Lorentz rescaling in Lemma \ref{lemma:Lorentz} and the induction hypothesis $Q_{\varepsilon, p, q}(R') \leq \bar C_{\varepsilon}$ with 
\[R'=R^{1+\delta}/(R^\delta K^2)=R/K^2 \leq R/2,\]
one has
\begin{multline*}
    \Big( \sum_\tau \sum_{\tau'\cap \tau \neq \varnothing } \|U_{[h;b]}   (f_{\tau}^\nu)_{\tau'}  \|_{L^q(  R^\delta B_R)}^q \Big)^{1/q} \\ \lesssim \bar C_{\varepsilon}  K^{\frac{n+1}{q} - \frac{n-1}{p}}  ( R/K^2)^{\beta + \varepsilon } \Big( \sum_{\tau} \sum_{\tau' \cap \tau \neq \varnothing} \| (\widetilde{f}_{\tau}^\nu)_{\tau'} \|_{L^p(\R^n)}^q \Big)^{1/q}.
\end{multline*}
By the embedding $\ell^p \subseteq \ell^q$ for $p \leq q$; the bounds
\[
(\sum_{\tau} \| f_{\tau}^\nu \|_p^p)^{1/p} \lesssim \| f^\nu \|_p
\quad \text{and} \quad
\Big( \sum_{\tau} \sum_{\tau' \cap \tau \neq \varnothing} \| (\widetilde{f}_{\tau}^\nu)_{\tau'} \|_{L^p(\R^n)}^p \Big)^{1/p} \lesssim \| f^\nu \|_p,\]
which follow by interpolation between $p=2$ and $p=\infty$; ii) of Lemma \ref{lemma:decomposition RH}, that is, $\|f^\nu \|_{L^p(\R^n)} \leq \| f \|_{L^p(\R^n)}$; and choosing $\tilde{N}$ sufficiently large, one has that the right-hand side of \eqref{eq:narrow triv dec} is controlled by
\begin{equation}\label{eq:narrow main}
 C \bar C_{\varepsilon} \max \{1, K^{(k-3)(\frac{1}{2} - \frac{1}{q} )}\} K^{\frac{n+1}{q} - \frac{n-1}{p} - 2\beta - \varepsilon + \delta'} R^{\delta\frac{(n-1)}{2}} R^{\beta + \varepsilon} \| f \|_{L^p(\R^n)}.
\end{equation}

\subsection*{Closing the induction}
By \eqref{eq:norm balanced dec}, the estimates \eqref{eq:broad main} and \eqref{eq:narrow main} for each $\nu=0, \dots, m \lfloor \log R \rfloor$, where $m > \frac{n+1}{q}+n$, and the error estimate \eqref{eq:error balanced}, one obtains
\begin{equation}
\label{eq:conclusion}
    \| U_{[h;b]} f  \|_{L^q (B_R )} \leq  \log R \cdot
    \left( \I + \II \right) \| f \|_{L^p(\R^n)} , 
\end{equation}
    where
    \begin{align*}
    \I &= D(\varepsilon) K^{D(\varepsilon)}  R^{\varepsilon/2} R^{(n-1)(\frac{1}{2}-\frac{1}{p})},\\
    \II &= D(\delta,\delta') \bar C_{\varepsilon} \max \{1, K^{(k-3)(\frac{1}{2} - \frac{1}{q} )}\} K^{\frac{n+1}{q} - \frac{n-1}{p} - 2\beta - \varepsilon + \delta'} R^{\delta \frac{(n-1)}{2}} R^{\beta + \varepsilon}
\end{align*}
and 
\[\frac{1}{p}-\frac{1}{\bar p_n}=(n+k+1) \Big( \frac{1}{2} - \frac{1}{\bar p_n} \Big) \Big(\frac{1}{p} - \frac{1}{q} \Big),\]
$2 \leq p \leq \bar p_n$, $\bar p_{n,k} \leq q \leq \bar p_n$. 
Here $D(\cdot)$ is a constant depending on the data as described after Definition \ref{def:qquantity} and the arguments in the parenthesis but not on $K$ and not on $R$. It is allowed to change from line to line.

We need to show that $\log R \cdot (\I + \II) \leq \bar C_\varepsilon R^{\beta+\varepsilon}$. This will require the exponent of $K$ in the second term of the right-hand side above to be negative. It is useful to note that
\begin{equation}\label{eq:exponent relation 1}
    \frac{n+1}{q} - \frac{n-1}{p} - 2\beta = \frac{n+1}{q} - \frac{n-1}{p'} -2 \Big( \frac{1}{q} - \bar \sigma \Big),
\end{equation}
as 
\[\beta=(n-1)\Big(\frac{1}{2}-\frac{1}{p}\Big) + \Big(\frac{1}{q}-\bar \sigma\Big).\]
We next analyse what choices of $p$, $q$ and $k$ allow us to close the induction and lead to sharp estimates on the critical line and off the critical line respectively.

\subsubsection*{\underline{Sharp regularity estimates on the critical line}} 
If $\bar \sigma=\sigma_{p,q}=1/q$ and $(1/p,1/q)$ is on the critical line
\begin{equation}
\label{eq:criticalline}
\frac{1}{q}=\frac{n-1}{n+1} \frac{1}{p'},
\end{equation}
the expression in \eqref{eq:exponent relation 1} equals to $0$. Thus, the exponent of $K$ in the second term of the right-hand side of \eqref{eq:conclusion} can only be negative if $k=2$ or $k=3$. Note that the critical line \eqref{eq:criticalline} meets the interpolation line 
\begin{equation}
    \label{eq:interpolationline}
    \frac{1}{p}-\frac{1}{\bar p_n}=(n+k+1) \Big( \frac{1}{2} - \frac{1}{\bar p_n} \Big) \Big(\frac{1}{p} - \frac{1}{q} \Big)
\end{equation}
at
\begin{align*}
p(k)&= \frac{2 \bar p_n \left(n^2+kn-1\right)-4 n (n+k+1)}{(n-1) \bar p_n (n+k+1)-2 (k (n-1)+n (n+1))},
\nonumber \\
\quad q(k)&= \frac{2  \bar p_n \left(n^2+kn-1\right)-4 n (n+k+1)}{(n-1)  \bar p_n (n+k-1)-2 (n-1) (k+n)},
\end{align*}
which satisfy $2 \leq p(k) \leq \bar p_n$, $\bar p_{n,k} \leq q(k) \leq \bar p_n$ for $2 \leq k \leq n+1$.
As $q(k)$ decreases with $k$, the best estimates are obtained using $k=3$, which is the highest possible $k$ that is still admissible. It is then our goal to show that
\begin{equation}\label{eq:p(3)q(3)}
        \| U_{[h;b]} f  \|_{L^{q(3)} (B_R)} \leq  \bar C_\varepsilon   R^{\beta + \varepsilon} \| f \|_{L^{p(3)}(\R^n)}.
\end{equation}

To this end, consider \eqref{eq:conclusion}
and use the bounds $\log R \lesssim_{\varepsilon} R^{\varepsilon/4}$ for the first term $\I$ and $\log R \lesssim_{\delta} R^{\delta (n-1)/2}$ for the second term $\II$; recall that $R \gtrsim_\varepsilon 1$. As $\beta=(n-1)(\frac{1}{2}-\frac{1}{p})$, the inequality \eqref{eq:conclusion} now reads 
    \begin{multline*}
        \| U_{[h;b]} f  \|_{L^{q(3)} (B_R)} \\ \leq  \Big( D (\varepsilon) K^{D(\varepsilon)} R^{\beta+3\varepsilon/4}  + D(\delta,\delta') \bar C_{\varepsilon}  K^{ \delta' - \varepsilon } R^{\delta(n-1)} R^{\beta + \varepsilon} \Big) \| f \|_{L^{p(3)}(\R^n)}.
    \end{multline*}
    Choose $\delta'=\varepsilon/2$ so that $K^{\delta'-\varepsilon}=K^{-\varepsilon/2}$. Then choose $K=K_0 R^{2\delta(n-1)/\varepsilon}$ for sufficiently large $K_0 \geq 1$, depending on $\delta$, $\varepsilon$ and the data so that $D(\delta,\delta')K_0^{-\varepsilon/2} \leq 1/2$. Then
    \begin{equation*}
        \| U_{[h;b]} f  \|_{L^{q(3)} (B_R)} \leq  \Big( D (\varepsilon) K^{D(\varepsilon)} R^{\beta-\varepsilon/4}  +  \bar C_{\varepsilon}/2 \Big) R^{\beta + \varepsilon}  \| f \|_{L^{p(3)}(\R^n)}
    \end{equation*}    
    with $K=K_0 R^{2\delta(n-1)/\varepsilon}$. We choose 
\[ 2 \delta =\min\Big\{\frac{\varepsilon^2}{4D(\varepsilon) (n-1)}, \frac{\varepsilon^2}{4(n-1)^2} \Big\}.\] 
Finally,  choose $\bar C_{\varepsilon}$ large enough so that
\[D(\varepsilon) K_0^{D(\varepsilon)} \leq \bar C_{\varepsilon}/2,\] which is admissible as the parameter $K_0$ only depends on $\varepsilon$ and the data. One then concludes that 
\begin{equation*}
        \| U_{[h;b]} f  \|_{L^{q(3)} (B_R)} \leq \bar C_{\varepsilon} R^{\beta + \varepsilon} \| f \|_{L^{p(3)}(\R^n)}
    \end{equation*} 
using the first value in the definition of $\delta$, which is the desired estimate \eqref{eq:p(3)q(3)}. This closes the induction provided we can verify the flatness condition $K^2 \leq \lambda^{\varepsilon/(n-1)}$ and the condition $K^{\varepsilon/2} \lesssim R^{\varepsilon^2/4}$ required by the broad estimate. Note that by the second entry in the definition of $\delta$, and using $R \leq \lambda$,
\[K^2=K_0^2 R^{4\delta(n-1)/\varepsilon}
\leq K_0^2 R^{\varepsilon/2(n-1)} 
\leq (K_0^2 \lambda^{-\varepsilon/2(n-1)}) \lambda^{\varepsilon/(n-1)} 
\leq \lambda^{\varepsilon/(n-1)}\] 
as $\lambda \gg 1$ and the parameter $K_0$ only depends on $\varepsilon$ and the data. 
Similarly, by the first entry in the definition of $\delta$, 
\[K^{\varepsilon/2}=K_0^{\varepsilon/2} R^{\delta(n-1)} \leq K_0^{\varepsilon/2} R^{\varepsilon^2/8} \lesssim R^{\varepsilon^2/4}\] 
as we are only concerned with $R \gtrsim_\varepsilon 1$. 

Therefore a sharp $(p(3),q(3),\sigma_{p(3), q(3)})$ local smoothing estimate holds, and a further interpolation with the elementary $(1,\infty,0)$ estimate yields the  estimates $(p,q,\sigma_{p,q})$ on the  critical line \eqref{eq:criticalline} for all $q \geq q(3)$.
%$= \frac{2  \bar p_n \left(n^2+3n-1\right)-4 n (n+4)}{(n-1) ( (n+2) \bar p_n - 2n -6  )}$. 
This proves Theorem \ref{thm:LpLq}.

\subsubsection*{ \underline{Sharp regularity estimates away from the critical line}} Consider first the case
\[\frac{1}{q} > \frac{n-1}{n+1} \frac{1}{p'}, \quad 2 \leq p \leq \frac{2n}{n-1}, \quad \bar \sigma = \sigma_{p,q}=\frac{(n-1)}{2}\Big(\frac{1}{p'}-\frac{1}{q}\Big),\]
with $p \leq q$, $p'<q$. For this data, the expression \eqref{eq:exponent relation 1} is identically zero for any pair $(p,q)$. Thus, the only possibilities for the exponent of $K$ to be negative in \eqref{eq:conclusion} are again $k=2$ and $k=3$. The estimates arising repeating the above analysis are implied by the interpolation of the sharp estimates $(p,q,\sigma_{p,q})$ for $q \geq q(3)$ in the critical line \eqref{eq:criticalline}
with the fixed-time estimate $p=q=2$. This proves the sharp bounds in the region $\mathfrak{T}_n \backslash \overline{P_1P_2}$ in Theorem \ref{thm:smoothing all k}.
    
Consider next the case
    \[\frac{1}{q} < \frac{n-1}{n+1} \frac{1}{p'}, \quad  \frac{2n}{n-1} < q, \quad \bar \sigma = \sigma_{p,q}=\frac{1}{q}, \quad p \leq q.\]
    Using the relation \eqref{eq:exponent relation 1}, one has that the exponent of $K$ in the second term in \eqref{eq:conclusion} is negative provided that
    \begin{equation}\label{eq:exponents 2}
    (k-3)\Big(\frac{1}{2}-\frac{1}{q}\Big) + \frac{n+1}{q} - \frac{n-1}{p'} \leq 0.
    \end{equation}
    We are thus allowed to use higher values for $k \geq 3$. For $(p,q)$ in the interpolation line \eqref{eq:interpolationline},
    the condition \eqref{eq:exponents 2} is saturated at $(p,q)=(\bar p (k), \bar q (k))$, where
    \begin{align*}
    \bar p(k) &= \frac{ 2 \bar p_n \left(2n^2 +k(n+4) -k^2+3 n-5\right)-4 \left(n+k+1\right)\left(2n-k+3\right)}{  \bar p_n \left(n+k+1\right)\left(2n-k+1\right) -2 \left(2n^2 +k(n-2) -k^2 +5 n+9\right)}
, \nonumber \\
\bar q(k) &= \frac{ 2 \bar p_n \left(2n^2 +k(n+4) -k^2+3 n-5\right)-4 \left(n+k+1\right)\left(2n-k+3\right)}{ \bar p_n \left(n+k-1\right)\left(2n-k+1\right) -2 \left( 2n^2 +kn -k^2+n+3\right)}.
    \end{align*}
The pair of exponents $(\bar p(k),\bar q(k))$ satisfies the constraints
\[2 \leq \bar p (k) \leq \bar p_n, \quad \bar p_{n,k} \leq \bar q (k) \leq \bar p_n, \quad \frac{2n}{n-1} < \bar q (k) < \infty\]
for any integer $2 \leq k \leq n+1$. Arguing as for the critical line, one can close the induction and obtain the sharp estimate
    $$
        \| U_{[h;b]} f  \|_{L^{\bar q(k)} (B_R)} \leq  \bar C_\varepsilon   R^{\beta + \varepsilon} \| f \|_{L^{\bar p(k)}(\R^n)}.
    $$
    This yields a set of $(p,q,\sigma_{p,q})$ local smoothing estimates for each $2 \leq k \leq n+1$, which can all be interpolated together with $(1/\bar p_n, 1 / \bar p_n)$ and the fixed time estimates at $P_0=(0,0)$, $P_1=(1,0)$ from \eqref{eq:LpLq fixed time Us} to yield sharp estimates for $(1/p,1/q) \in \mathfrak{P}_n \backslash \overline{P_0P_1}$. This proves Theorem \ref{thm:smoothing all k}.

\begin{remark}
Despite the focus of this paper on sharp regularity local smoothing estimates, 
it is also natural to explore
what are the (not necessarily sharp) regularity estimates
that the use of higher degrees of multilinearity would imply on the critical line \eqref{eq:criticalline}. In order to close the induction in the proof of Theorem \ref{thm:LpLq}, one requires the exponent of $K$ in the second term in the right-hand side of \eqref{eq:conclusion} to be negative. By \eqref{eq:exponent relation 1}, this requires
    $$
    \bar \sigma \leq \sigma(k):= \frac{k-1}{2q(k)} - \frac{k-3}{4}.
    $$
Note that $0 < \sigma(k) \leq 1/q(k)$ if $2 \leq q(k) < \frac{2(k-1)}{k-3}$. Controlling $R^{(n-1)(\frac{1}{2}-\frac{1}{p})} \leq R^{\beta}$, one can then argue as above to obtain, for each fixed $k$, a $(p(k),q(k),\sigma(k))$ local smoothing estimate. However, with the input $\bar p_n=\frac{2(n+1)}{n-1}$, such estimates are worse than those obtained by interpolation of the case $k=3$ and the known local smoothing estimates for all $\sigma < 1/(2p)$ at $p=q=\frac{2n}{n-1}$, which themselves follow by interpolation from the sharp estimates at $\bar p_n=\frac{2(n+1)}{n-1}$ and $L^2$. Indeed, that interpolation yields $(p,q,\sigma^*)$ estimates on the critical line \eqref{eq:criticalline} with
\[\sigma^*=\frac{n+5}{4} - \frac{n^2+4n+1}{2q(n-1)}, \quad \frac{2n}{n-1} \leq q \leq \frac{2(n^2+6n-1)}{(n-1)(n+5)}.\]
One can hence verify that for $q=q(k)$, one has $\sigma^*>\sigma(k)$ if $k>3$. Better non-sharp regularity results could be obtained by interpolation with the most recent results at $p=q=\frac{2n}{n-1}$ from \cite{GLMX}. The computation is left to the interested reader.
\end{remark}

\bibliography{Reference}

\begin{thebibliography}{10}

\bibitem{Bejenaru2020}
I.~Bejenaru.
\newblock The almost optimal multilinear restriction estimate for hypersurfaces
  with curvature: the case of $n-1$ hypersurfaces in $\mathbb{R}^n$.
\newblock To appear in \textit{International Mathematics Research Notices,
  IMRN}, doi: 10.1093/imrn/rnab208, Preprint: \verb+arXiv:2002.12488+.

\bibitem{BRS2018}
D.~Beltran, J.~P. Ramos, and O.~Saari.
\newblock Regularity of fractional maximal functions through {F}ourier
  multipliers.
\newblock {\em J. Funct. Anal.}, 276(6):1875--1892, 2019.

\bibitem{BCT}
J.~Bennett, A.~Carbery, and T.~Tao.
\newblock On the multilinear restriction and {K}akeya conjectures.
\newblock {\em Acta Math.}, 196(2):261--302, 2006.

\bibitem{BD2015}
J.~Bourgain and C.~Demeter.
\newblock The proof of the {$l^2$} decoupling conjecture.
\newblock {\em Ann. of Math. (2)}, 182(1):351--389, 2015.

\bibitem{BG}
J.~Bourgain and L.~Guth.
\newblock Bounds on oscillatory integral operators based on multilinear
  estimates.
\newblock {\em Geom. Funct. Anal.}, 21(6):1239--1295, 2011.

\bibitem{Brenner}
P.~Brenner.
\newblock On {$L_{p}-L_{p^{\prime} }$} estimates for the wave-equation.
\newblock {\em Math. Z.}, 145(3):251--254, 1975.

\bibitem{Cordoba82}
A.~C\'{o}rdoba.
\newblock Geometric {F}ourier analysis.
\newblock {\em Ann. Inst. Fourier (Grenoble)}, 32(3):vii, 215--226, 1982.

\bibitem{GLMX}
C.~Gao, B.~Liu, C.~Miao, and Y.~Xi.
\newblock Improved local smoothing estimate for the wave equation in higher
  dimensions.
\newblock Preprint: {\verb+arxiv.org/abs/2108.06870+} (2021).

\bibitem{GMZ}
C.~Gao, C.~Miao, and J.~Zheng.
\newblock Improved local smoothing estimates for the fractional
  {S}chr\"{o}dinger operator.
\newblock {\em Bull. Lond. Math. Soc.}, 54(1):54--70, 2022.

\bibitem{Garrigos2009}
G.~Garrig\'os and A.~Seeger.
\newblock On plate decompositions of cone multipliers.
\newblock {\em Proc. Edinb. Math. Soc. (2)}, 52(3):631--651, 2009.

\bibitem{Garrigos2010}
G.~Garrig\'os and A.~Seeger.
\newblock A mixed norm variant of {W}olff's inequality for paraboloids.
\newblock In {\em Harmonic analysis and partial differential equations}, volume
  505 of {\em Contemp. Math.}, pages 179--197. Amer. Math. Soc., Providence,
  RI, 2010.

\bibitem{Guth_lecture_notes}
L.~Guth.
\newblock Decoupling seminar notes.
\newblock {\verb+http://math.mit.edu/~lguth/decouplingseminar/+}.

\bibitem{Guth14}
L.~Guth.
\newblock A restriction estimate using polynomial partitioning.
\newblock {\em J. Amer. Math. Soc.}, 29(2):371--413, 2016.

\bibitem{Guth16}
L.~Guth.
\newblock Restriction estimates using polynomial partitioning {II}.
\newblock {\em Acta Math.}, 221(1):81--142, 2018.

\bibitem{GHI}
L.~Guth, J.~Hickman, and M.~Iliopoulou.
\newblock Sharp estimates for oscillatory integral operators via polynomial
  partitioning.
\newblock {\em Acta Math.}, 223(2):251--376, 2019.

\bibitem{GWZ}
L.~Guth, H.~Wang, and R.~Zhang.
\newblock A sharp square function estimate for the cone in {$\mathbb{R}^3$}.
\newblock {\em Ann. of Math. (2)}, 192(2):551--581, 2020.

\bibitem{HKL2021}
S.~Ham, H.~Ko, and S.~Lee.
\newblock Circular average relative to fractal measures.
\newblock to appear in \textit{Communications on Pure and Applied Analysis},
  doi: 10.3934/cpaa.2022100, Preprint: {\verb+arxiv.org/abs/2110.11185+}
  (2021).

\bibitem{Harris18}
T.~L.~J. Harris.
\newblock Improved decay of conical averages of the {F}ourier transform.
\newblock {\em Proc. Amer. Math. Soc.}, 147(11):4781--4796, 2019.

\bibitem{HNS2011}
Y.~Heo, F.~Nazarov, and A.~Seeger.
\newblock Radial {F}ourier multipliers in high dimensions.
\newblock {\em Acta Math.}, 206(1):55--92, 2011.

\bibitem{Laba2002}
I.~{\L}aba and T.~Wolff.
\newblock A local smoothing estimate in higher dimensions.
\newblock {\em J. Anal. Math.}, 88:149--171, 2002.
\newblock Dedicated to the memory of Tom Wolff.

\bibitem{LeeLS}
J.~Lee.
\newblock A trilinear approach to square function and local smoothing estimates
  for the wave operator.
\newblock {\em Indiana Univ. Math. J.}, 69(6):2005--2033, 2020.

\bibitem{Lee2003}
S.~Lee.
\newblock Endpoint estimates for the circular maximal function.
\newblock {\em Proc. Amer. Math. Soc.}, 131(5):1433--1442, 2003.

\bibitem{Lee2006}
S.~Lee.
\newblock Linear and bilinear estimates for oscillatory integral operators
  related to restriction to hypersurfaces.
\newblock {\em J. Funct. Anal.}, 241(1):56--98, 2006.

\bibitem{Lee2013}
S.~Lee and A.~Seeger.
\newblock Lebesgue space estimates for a class of {F}ourier integral operators
  associated with wave propagation.
\newblock {\em Math. Nachr.}, 286(7):743--755, 2013.

\bibitem{Lee2012}
S.~Lee and A.~Vargas.
\newblock On the cone multiplier in {$\Bbb{R}^3$}.
\newblock {\em J. Funct. Anal.}, 263(4):925--940, 2012.

\bibitem{Littman}
W.~Littman.
\newblock {$L^{p}-L^{q}$}-estimates for singular integral operators arising
  from hyperbolic equations.
\newblock In {\em Partial differential equations ({P}roc. {S}ympos. {P}ure
  {M}ath., {V}ol. {XXIII}, {U}niv. {C}alifornia, {B}erkeley, {C}alif., 1971)},
  pages 479--481, 1973.

\bibitem{MiyachiWave}
A.~Miyachi.
\newblock On some estimates for the wave equation in {$L^{p}$}\ and {$H^{p}$}.
\newblock {\em J. Fac. Sci. Univ. Tokyo Sect. IA Math.}, 27(2):331--354, 1980.

\bibitem{Mock93}
G.~Mockenhaupt.
\newblock A note on the cone multiplier.
\newblock {\em Proc. Amer. Math. Soc.}, 117(1):145--152, 1993.

\bibitem{Mockenhaupt1992}
G.~Mockenhaupt, A.~Seeger, and C.~D. Sogge.
\newblock Wave front sets, local smoothing and {B}ourgain's circular maximal
  theorem.
\newblock {\em Ann. of Math. (2)}, 136(1):207--218, 1992.

\bibitem{OW}
Y.~Ou and H.~Wang.
\newblock A cone restriction estimate using polynomial partitioning.
\newblock {\em J. Eur. Math. Soc. (JEMS)}, 24(10):3557--3595, 2022.

\bibitem{Peral1980}
J.~C. Peral.
\newblock {$L^{p}$}\ estimates for the wave equation.
\newblock {\em J. Funct. Anal.}, 36(1):114--145, 1980.

\bibitem{Schippa}
R.~Schippa.
\newblock Oscillatory integral operators with homogeneous phase functions.
\newblock {P}reprint: {\verb+arXiv:2109.14040+} (2021).

\bibitem{SS1997}
W.~Schlag and C.~D. Sogge.
\newblock Local smoothing estimates related to the circular maximal theorem.
\newblock {\em Math. Res. Lett.}, 4(1):1--15, 1997.

\bibitem{Sogge91}
C.~D. Sogge.
\newblock Propagation of singularities and maximal functions in the plane.
\newblock {\em Invent. Math.}, 104(2):349--376, 1991.

\bibitem{bigStein}
E.~M. Stein.
\newblock {\em Harmonic analysis: real-variable methods, orthogonality, and
  oscillatory integrals}, volume~43 of {\em Princeton Mathematical Series}.
\newblock Princeton University Press, Princeton, NJ, 1993.
\newblock With the assistance of Timothy S. Murphy, Monographs in Harmonic
  Analysis, III.

\bibitem{Strichartz1970}
R.~S. Strichartz.
\newblock Convolutions with kernels having singularities on a sphere.
\newblock {\em Trans. Amer. Math. Soc.}, 148:461--471, 1970.

\bibitem{Strichartz77}
R.~S. Strichartz.
\newblock Restrictions of {F}ourier transforms to quadratic surfaces and decay
  of solutions of wave equations.
\newblock {\em Duke Math. J.}, 44(3):705--714, 1977.

\bibitem{TaoBR}
T.~Tao.
\newblock The {B}ochner-{R}iesz conjecture implies the restriction conjecture.
\newblock {\em Duke Math. J.}, 96(2):363--375, 1999.

\bibitem{TaoCone}
T.~Tao.
\newblock Endpoint bilinear restriction theorems for the cone, and some sharp
  null form estimates.
\newblock {\em Math. Z.}, 238(2):215--268, 2001.

\bibitem{TV2}
T.~Tao and A.~Vargas.
\newblock A bilinear approach to cone multipliers. {II}. {A}pplications.
\newblock {\em Geom. Funct. Anal.}, 10(1):216--258, 2000.

\bibitem{Wolff2000}
T.~Wolff.
\newblock Local smoothing type estimates on {$L^p$} for large {$p$}.
\newblock {\em Geom. Funct. Anal.}, 10(5):1237--1288, 2000.

\bibitem{WolffCone}
T.~Wolff.
\newblock A sharp bilinear cone restriction estimate.
\newblock {\em Ann. of Math. (2)}, 153(3):661--698, 2001.

\end{thebibliography}

\bibliographystyle{abbrv}

\end{document}